\documentclass[11pt]{amsart}
\usepackage[top=90pt,bottom=90pt,left=90pt,right=90pt]{geometry}
\usepackage{amssymb}
\usepackage{amsmath}
\usepackage{youngtab}
\usepackage{young}
\usepackage{xfrac}
\usepackage{makecell}
\usepackage{ytableau}
\usepackage{tikz-cd}
\usepackage{amsfonts,mathrsfs}
\usepackage{bbm,amsmath,amssymb,amsfonts,graphicx,color,subfig,mathtools, verbatim,MnSymbol}
\usepackage[bookmarksopen=false,pdftex=true,breaklinks=true,%
      backref=page,pagebackref=true,plainpages=false,%
      hyperindex=true,pdfstartview=FitH,colorlinks=true,%
      pdfpagelabels=true,colorlinks=true,linkcolor=blue,%
      citecolor=red,urlcolor=green,hypertexnames=false%
      ]%
   {hyperref}

\newtheorem{theorem}{Theorem}[section]
\newtheorem{lemma}[theorem]{Lemma}
\newtheorem{corollary}[theorem]{Corollary}
\newtheorem{proposition}[theorem]{Proposition}
\theoremstyle{definition}
\newtheorem{definition}[theorem]{Definition}
\newtheorem{remark}[theorem]{Remark}

\newtheorem{question}[theorem]{Question}

\newcommand{\ZZ}{\mathbb{Z}}
\newcommand{\QQ}{\mathbb{Q}}
\newcommand{\RR}{\mathbb{R}}
\newcommand{\CC}{\mathbb{C}}
\newcommand{\PP}{\mathbb{P}}
\newcommand{\mm}{\mathfrak{m}}
\newcommand{\HH}{\mathcal{H}}

\DeclareMathOperator{\sh}{sh}

\DeclareMathOperator{\len}{len}
\DeclareMathOperator{\spe}{sp}

\DeclareMathOperator{\Spec}{Spec}
\DeclareMathOperator{\Hilb}{Hilb}
\DeclareMathOperator{\ev}{ev}
\DeclareMathOperator{\sgn}{sgn}
\DeclareMathOperator{\Hom}{Hom}
\DeclareMathOperator{\gr}{gr}

\DeclareFontFamily{U}{mathx}{\hyphenchar\font45}
\DeclareFontShape{U}{mathx}{m}{n}{
      <5> <6> <7> <8> <9> <10>
      <10.95> <12> <14.4> <17.28> <20.74> <24.88>
      mathx10
      }{}
\DeclareSymbolFont{mathx}{U}{mathx}{m}{n}
\DeclareFontSubstitution{U}{mathx}{m}{n}
\DeclareMathAccent{\widecheck}{0}{mathx}{"71}

\begin{document}
\nocite*{}
\title[]
{
Symmetric Ideals and Invariant Hilbert Schemes
}

\author{Sebastian Debus}
\address{Technische Universität Chemnitz, Fakultät für Mathematik, 09107 Chemnitz, Germany}
\email{sebastian.debus@math.tu-chemnitz.de}
\author{Andreas Kretschmer}
\address{Humboldt-Universität zu Berlin, Institut für Mathematik, Rudower Chaussee 25,
12489 Berlin, Germany}
\email{andreas.kretschmer@hu-berlin.de}

\begin{abstract}
A symmetric ideal is an ideal in a polynomial ring which is stable under all permutations of the variables. In this paper we initiate a global study of zero-dimensional symmetric ideals. By this we mean a geometric study of the invariant Hilbert schemes $\Hilb_{\rho}^{S_n}(\CC^n)$ parametrizing symmetric subschemes of $\CC^n$ whose coordinate rings, as $S_n$-modules, are isomorphic to a given representation $\rho$. In the case that $\rho = M^\lambda$ is a permutation module corresponding to certain special types of partitions $\lambda$ of $n$, we prove that $\Hilb_{\rho}^{S_n}(\CC^n)$ is irreducible or even smooth. We also prove irreducibility whenever $\dim \rho \leq 2n$ and the invariant Hilbert scheme is non-empty. In this same range, we classify all homogeneous symmetric ideals and decide which of these define singular points of $\Hilb_{\rho}^{S_n}(\CC^n)$. A central tool is the combinatorial theory of higher Specht polynomials. 
\end{abstract}

\maketitle

\section{Introduction}

A symmetric ideal is an ideal in a polynomial ring which is stable under all permutations of the variables.
Special classes of symmetric ideals, for instance \emph{Specht ideals} and \emph{Tanisaki ideals} \cite{Tanisaki1982Defining,BergeronGarsia1992Harmonic,GarsiaProcesi1992}, have been studied intensively in the algebraic combinatorics literature and are related to Kostka--Macdonald polynomials and the famous work of Haiman on $n!$, see \cite{Haiman2003Combinatorics} and the references therein.
In commutative algebra, symmetric ideals are mainly studied for their asymptotic properties. We refer to \cite{Nagel2017Equivariant,Nagel2019FIvaryingCoeff,NagelRoemer2020CodimUpToSymmetry,NagelRoemer2021RegularityUpToSymmetry,nagpal2021symmetric} for plenty of examples as well as to \cite{ChurchFarb2013RepresentationStability,ChurchEllenbergFarb2015FImodules} for foundational results on representation stability and FI-modules. One of the most well-known results on symmetric ideals is that an infinite chain of symmetric ideals $I_n \subseteq \CC[x_1,\ldots,x_n]$ in increasingly many variables with $I_{n-1} \subseteq I_n$ for all $n$ eventually stabilizes in the sense that $I_n = (S_n \cdot I_{n-1})$ for all $n$ large enough \cite{Cohen1967Laws,Aschenbrenner2007Finite,Hillar2012Finite,Draisma2014Noetherianity}, where $S_n$ is the symmetric group on $n$ elements.

The main goal of this article is to initiate a global study of symmetric ideals, starting with the zero-dimensional case. By this we mean a geometric study of the invariant Hilbert schemes $\Hilb_{\rho}^{S_n}(\CC^n)$ parametrizing symmetric subschemes of $\CC^n$ whose coordinate rings are isomorphic to $\rho$ as $S_n$-representations. This is a very special case of the invariant Hilbert schemes of Alexeev and Brion \cite{AlexeevBrion,BrionInvariantHilb}.

Of central interest to us is the case where $\rho = M^\lambda$ is the permutation module corresponding to a partition $\lambda$ of $n$. It is a consequence of \cite{BergeronGarsia1992Harmonic} that there is only a single homogeneous ideal in $\Hilb_{M^\lambda}^{S_n}(\CC^n)$, called the Tanisaki ideal $I_\lambda$. It can be obtained as the associated graded ideal of the vanishing ideal of the $S_n$-orbit of a point in $\CC^n$ whose entries occur with multiplicities $\lambda_1, \lambda_2,$ and so on \cite{GarsiaProcesi1992}. The closure of the set of these vanishing ideals is an irreducible component of $\Hilb_{M^\lambda}^{S_n}(\CC^n)$, the \emph{smoothable component}. A natural question is therefore whether this is all there is, i.e., whether $\Hilb_{M^\lambda}^{S_n}(\CC^n)$ is irreducible. For certain types of partitions, for instance if $\lambda$ is a hook partition or has only two non-zero parts, we show that $\Hilb_{M^\lambda}^{S_n}(\CC^n)$ is indeed irreducible, see Corollary~\ref{corollary:hookPartition}. In the latter case it is even smooth by Proposition~\ref{prop:m=2}. For arbitrary $\lambda$, we prove that the Hilbert--Chow morphism
\begin{equation*}
    \gamma \colon \Hilb_{M^\lambda}^{S_n}(\CC^n) \longrightarrow \CC^n/S_n,
\end{equation*}
mapping an ideal to its vanishing orbit, is finite and has singleton fibers over the image of the smoothable component (Theorem~\ref{thm:singleOrbit}). From this one may deduce that the dimension of $\Hilb_{M^\lambda}^{S_n}(\CC^n)$ agrees with the number $m$ of non-zero parts of $\lambda$, and that the normalization of its smoothable component is isomorphic to $\CC^m$ with the normalization map being bijective.

In addition, for all $\rho$ satisfying $\dim \rho \leq 2n$ we prove that $\Hilb_\rho^{S_n}(\CC^n)$ is irreducible whenever it is non-empty, and we decide when this happens (Theorem~\ref{thm:classification}). Moreover, in the same range we explicitly classify all homogeneous ideals in $\Hilb_\rho^{S_n}(\CC^n)$ and decide which of these define singular points, see Table~\ref{tab:classification}.

Beyond these cases, it can be observed that if $\rho = \CC[S_n]^{\oplus l}$ is a direct sum of regular representations, the Hilbert--Chow morphism is actually an isomorphism $\Hilb_{\CC[S_n]^{\oplus l}}^{S_n}(\CC^n) \cong \Hilb_l(\CC^n/S_n)$ (Proposition~\ref{prop:HilbChowIso}). In particular, this shows that $\Hilb^{S_n}_\rho(\CC^n)$ is in general at least as complicated as the usual Hilbert scheme of points of affine space. An analogous statement should also hold for other types of finite reflection groups.

A peculiar observation in our study is that the geometry of $\Hilb_{\rho}^{S_{n}}(\CC^{n})$ does not seem to depend strongly on the ambient dimension $n$ when $n$ grows while $\rho$ ``stays the same,'' in an appropriate sense. For instance, this is what made it possible to complete our classification in the range $\dim \rho \leq 2n$ for arbitrary $n$. We make this apparent stabilization phenomenon more precise in Subsection~\ref{subsec:stabilization}, see Question~\ref{question:stabilization}.

The paper is structured as follows. In Section~\ref{section:preliminaries} we quickly introduce higher Specht polynomials and invariant Hilbert schemes. Section~\ref{section:Tanisaki} deals with the case of permutation modules, $\rho = M^\lambda$. The classification in the range $\dim \rho \leq 2n$ is contained in Section~\ref{section:classification}. Several open questions are raised in the final Section~\ref{section:conclusion}.

\section{Preliminaries}\label{section:preliminaries}

\subsection{Higher Specht polynomials}
We encourage the reader familiar with basic notions of the characteristic~zero representation theory of $S_n$ to skip this subsection on a first reading and refer back to it when needed.

We introduce our notation and define the \emph{higher Specht polynomials} of \cite{Terasoma1993Higher}. Foundational references on $S_n$-representations and related combinatorics are, e.g., \cite{FultonHarris1991ReprTheory,Fulton,Sagan}.

A \emph{partition} $\lambda=(\lambda_1,\ldots,\lambda_m)$ of $n$ is a sequence of non-increasing positive integers such that $\sum_{i\geq1} \lambda_i = n$ and $\lambda_m \neq 0$.
We write $\lambda \vdash n$.
The \emph{size} of $\lambda$ is $|\lambda| \coloneqq n$ and its \emph{length} is denoted by $\len(\lambda) \coloneqq m$. 

If $\lambda, \mu \vdash n$, then $\mu$ \emph{dominates} $\lambda$ if $\sum_{j=1}^k \mu_j \geq \sum_{j=1}^k \lambda_j$ for all $k$, where $\mu_j$ is interpreted as zero if $j > \len(\mu)$. We denote domination by $\mu \unrhd \lambda$.
A partition $\lambda$ can be represented by its \emph{(Young) diagram}, i.e., the ordered sequence of left-justified boxes from the left to the right and the top to the bottom, where the $i$-th line contains $\lambda_i$ boxes. We say that the associated diagram has \emph{shape} $\lambda$. The \emph{transpose partition} $\lambda'$ of $\lambda$ is the partition whose diagram is obtained from transposing the diagram of $\lambda$. A \emph{tableau} of shape $\lambda$ is a bijective filling of the diagram of shape $\lambda$ with the integers in $[n] \coloneqq \{1,\ldots,n\}$. We write $\sh(T) = \lambda$ if $T$ is a tableau of shape $\lambda$.
For instance,
$$\ytableausetup{smalltableaux}
\begin{ytableau}
9 & 3 & 6 & 4 \\
2 & 1 & 8\\ 
5 & 7
\end{ytableau}$$
is a tableau of shape $(4,3,2)$.  
A \emph{standard tableau} is a tableau in which the integers in every row and column are increasing from the left to the right and the top to the bottom.

For a sequence $(i_1,\ldots,i_l)$ of distinct positive integers, we define the associated \emph{Vandermonde polynomial} in the variables $x_{i_1},\ldots,x_{i_l}$ as \[ \Delta_{(i_1,\ldots,i_l)}(x) = \prod_{j < k \in  [l]}(x_{i_j}-x_{i_k}). \] By convention, $\Delta_{(i)} = 1$.
For a tableau $T$ of shape $\lambda$ with $l$ columns, we denote by $T_i$ the sequence of numbers contained in the $i$-th column of $T$ from above to below. Then the associated \emph{Specht polynomial} $\spe_T(x)$ is the product of the Vandermonde polynomials of all columns $T_i$, i.e.,
\[ \spe_{T}(x) \coloneqq \prod_{i=1}^l \Delta_{T_i}.\]

For example, for the above tableau $T$ of shape $(4,3,2)$, we have 
\begin{align*}
\spe_{T}(x)&=\Delta_{(9,2,5)}(x)\Delta_{(3,1,7)}(x)\Delta_{(6,8)}(x)\Delta_{(4)}(x) 
\\ &=(x_9-x_2)(x_9-x_5)(x_2-x_5)(x_3-x_1)(x_3-x_7)(x_1-x_7)(x_6-x_8).
\end{align*}

For a group $G$ acting linearly on a complex vector space $V$ we denote the action of a group element $\sigma \in G$ on a vector $v \in V$ by $\sigma \cdot v$.
If $G$ is finite, then the irreducible representations are in correspondence with the conjugacy classes. In the case of the symmetric group $S_n$ these in turn correspond bijectively to integer partitions of $n$. For $\lambda \vdash n$ the irreducible representation of $S_n$ corresponding to $\lambda$ is the \emph{Specht module} $S^\lambda$. It can be constructed as the vector space spanned by all Specht polynomials of shape $\lambda$. It was proven by Specht \cite{Specht} that these representations are indeed irreducible and pairwise non-isomorphic.

We will consider the action of $S_n$ on the polynomial ring $P \coloneqq \CC[x_1,\ldots,x_n]$ via permutations of the variables $\sigma \cdot x_i \coloneqq x_{\sigma(i)}$.
A polynomial is \textit{symmetric} if it is invariant under the $S_n$-action. The invariant ring $P^{S_n}$ is generated by $n$ algebraically independent polynomials which can be chosen as the sequence of the first $n$ elementary symmetric polynomials $e_j= \sum_{I \subset [n], |I| = n} \prod_{i \in I} x_i$ or the power sum polynomials $p_j = \sum_{i=1}^n x_i^j$. The \emph{coinvariant algebra} of $P$ is the quotient of $P$ by the ideal $(p_1,\ldots,p_n)$ generated by all homogeneous symmetric polynomials of positive degree. As an $S_n$-representation, the coinvariant algebra has the structure of the \emph{regular representation} of $S_n$, i.e., each irreducible representation $S^\lambda$ occurs $\dim(S^\lambda)$ times in $P / (p_1,\ldots,p_n)$. Moreover, we have an isomorphism of graded $S_n$-modules $P \cong P/(p_1,\ldots,p_n) \otimes_\CC \CC[p_1,\ldots,p_n]$. An isotypic decomposition of the coinvariant algebra lifts to an isotypic decomposition of $P$ by multiplying the coinvariants by symmetric polynomials. Terasoma and Yamada \cite{Terasoma1993Higher} provided an explicit decomposition of the coinvariant algebra into its irreducible subspaces by introducing \emph{higher Specht polynomials}. This was extended in \cite{Ariki1997Higher}, for example to certain direct products of reflection groups.

For a Young tableau $T$ of shape $\lambda \vdash n$ we denote by $C(T)$ the \emph{column stabilizer} and by $R(T)$ the \emph{row stabilizer} of $T$. These are the maximal subgroups of $S_n$ which permute the entries in any column, respectively any row of $T$. 
The \emph{word} of $T$ is the sequence $w(T) \in \ZZ^n$ where we read each column of the tableau $T$ from the bottom to the top, starting from the left.
The \emph{index} $i(T) \in \ZZ^n$ of $T$ is defined as follows. The number $1$ in the word $w(T)$ has index $0$. Recursively, if the number $k$ in $w(T)$ has index $p$, then $k+1$ has index either $p$ or $p+1$ according to whether $k+1$ is right or left from $k$ in $w(T)$. For instance, for the above tableau $T$ we have $w(T)=(5,2,9,7,1,3,8,6,4)$ and $i(T)=(2,1,4,3,0,1,3,2,1)$.
Next, we can associate to a tuple of tableaux $(T,S)$ of the same shape $\lambda$ a monomial $x_T^{i(S)}$ in $n$ variables by
\begin{equation*}
    x_T^{i(S)} := x_{w(T)_1}^{i(S)_1} \cdots x_{w(T)_{n}}^{i(S)_{n}}.
\end{equation*}
The \emph{higher Specht polynomial} associated with $(T,S)$ is then defined as
$$F_{T}^S := \sum_{\substack{\sigma \in C(T)\\ \tau \in R(T)}} \sgn (\sigma) \sigma \circ \tau \cdot x_{T}^{i(S)} \in P.$$

\begin{theorem}[{\cite{Terasoma1993Higher}}]
The $\lambda$-part of the isotypic decomposition of the coinvariant algebra has a basis given by the higher Specht polynomials $F_T^S$ where $\sh(S) = \sh(T) = \lambda$ and $S,T$ are both standard tableaux. Moreover, for a fixed standard tableau $S$, the representation $\langle F_T^S : \sh(T) = \sh(S) = \lambda \rangle$ is isomorphic to the Specht module $S^{\lambda}$.
\end{theorem}

For instance, when $n=3$ we have 
\begin{align*}
  \CC[x_1,x_2,x_3]/(p_1,p_2,p_3) & \cong S^{(3)} \oplus 2 S^{(2,1)} \oplus S^{(1,1,1)} \\
  & = \langle 1 \rangle \oplus \langle x_i - x_j \rangle \oplus \langle (x_i-x_j)x_k \rangle \oplus \langle (x_1-x_2)(x_1-x_3)(x_2-x_3) \rangle.
\end{align*}
For any $\lambda$ there exists a unique standard tableau $S_0$ of shape $\lambda$ such that the sum of the entries of the indices of standard tableaux of shape $\lambda$ is minimal. The higher Specht polynomial $F_T^{S_0}$ then coincides with the Specht polynomial $\spe_T$ up to a multiplicative constant.

For a partition $\lambda \vdash n$ we denote by $O(\lambda)$ the subset of $\CC^n$ of points of orbit-type $\lambda$, i.e., the points whose stabilizer subgroup in $S_n$ is a conjugate of the Young subgroup $S_\lambda = S_{\lambda_1} \times \cdots \times S_{\lambda_m}$. The ideal of $P$ generated by all Specht polynomials of shape $\lambda$ is called the \emph{Specht ideal} of $\lambda$ and we write $V_\lambda$ for its vanishing set. Then $V_\lambda = \bigcup_{\mu \not \unlhd \lambda} O(\mu)$ by \cite[Corollary~1]{moustrou2021symmetric}. A simple consequence is $\bigcap_{\lambda \not \ntrianglerighteq \theta} V_\lambda = \bigcup_{\tau \trianglerighteq \theta} O(\tau)$. By abuse of notation, we denote by $O(\lambda)/S_n$ the image of $O(\lambda)$ in $\CC^n/S_n = \Spec(\CC[p_1,\ldots,p_n])$ and by $\overline{O(\lambda)}/S_n$ its closure. We emphasize that $\overline{O(\lambda)}/S_n$ is usually a non-normal subvariety of $\CC^n/S_n$.
Let us also recall the following well-known lemma.

\begin{lemma}
Set-theoretically, the fiber over every closed point of the map $(p_1,\ldots,p_n) : \CC^n \to \CC^n$ consisting of the first $n$ power sum polynomials is the $S_n$-orbit of a single point in $\CC^n$. In particular, the ideal $(p_1,\ldots,p_n) \subseteq P$ is a complete intersection.
\end{lemma}

\begin{proof}
We first consider the fiber of $(a_1,\dots,a_n) \in \CC^n$ for the map $(e_1,\ldots,e_n): \CC^n \to \CC^n$ consisting of the elementary symmetric polynomials. By Vieta's formula, we have $t^n - a_1 t^{n-1} + \ldots + (-1)^n a_n = \prod_{i=1}^n (t - \alpha_i)$ for some $\alpha_i \in \CC$ which are unique up to permutation. This shows that the preimage of $(a_1,\dots,a_n)$ consists precisely of the permutations of $(\alpha_1,\ldots,\alpha_n)$.
The same is then true for the map $(p_1,\ldots,p_n)$ by the Newton identities, which provide polynomial relations $p_j=(-1)^{j-1}je_{j}+\sum _{i=1}^{j-1}(-1)^{j-1+i}e_{j-i}p_i$, for all $1 \leq j \leq n$, between the power sum polynomials and the elementary symmetric polynomials.
\end{proof}

An important representation for us will be the permutation module $M^\lambda$. It can be defined as the the induced representation of the trivial representation for the Young subgroup corresponding to $\lambda$, in formulas $M^\lambda = \mathrm{Ind}^{S_n}_{S_{\lambda_1} \times \cdots \times S_{\lambda_m}} \mathbf{1}$, where $\mathbf{1}$ denotes the trivial representation and ${S_{\lambda_1} \times \cdots \times S_{\lambda_m}} \subset S_n$ is a short notation for the Young subgroup $S_{\{1,2,\ldots ,\lambda _{1}\}}\times S_{\{\lambda _{1}+1,\lambda _{1}+2,\ldots ,\lambda _{1}+\lambda _{2}\}}\times \cdots \times S_{\{n-\lambda _{\ell }+1,n-\lambda _{\ell }+2,\ldots ,n\}}$ of $S_n$. This can be seen to agree with the permutation action of $S_n$ on the quotient $S_n /S_{\lambda_1} \times \cdots \times S_{\lambda_m} $.

\subsection{Invariant Hilbert schemes}
For a reductive group $G$, an affine $G$-scheme $W = \Spec(A)$ of finite type over $\CC$, and a multiplicity-finite (but not necessarily finite-dimensional) $G$-representation $\rho$, the \emph{invariant Hilbert scheme} $\Hilb_\rho^G(W)$ parametrizes all closed $G$-subschemes $Z \subseteq W$ with $\Gamma(\mathcal{O}_Z) \cong \rho$ as $G$-representations. We refer to~\cite{AlexeevBrion,BrionInvariantHilb}, noting that the second reference allows $G$ to be disconnected.

The construction in the special case where $G$ is a finite group is quite simple: In this case, there are only finitely many irreducible $G$-representations up to isomorphism, so $\rho$ is automatically finite-dimensional. Hence, the parametrized subschemes $Z \subseteq W$ are zero-dimensional of length $\dim(\rho)$. It turns out that $\Hilb_\rho^G(W)$ can be constructed as the union of some connected components of the $G$-fixed point subscheme of $\Hilb_{\dim(\rho)}(W)$, the latter denoting the usual Hilbert scheme of $\dim(\rho)$ points on~$W$. Since the latter is quasi-projective, so is $\Hilb_\rho^G(W)$.

We collect several foundational results of~\cite{BrionInvariantHilb} valid for any reductive group $G$.

\begin{proposition}[{\cite[Proposition~3.5]{BrionInvariantHilb}}]
    Let $[Z] \in \Hilb_\rho^G(W)$ be a closed point corresponding to a $G$-stable closed subscheme $Z \subseteq W$ defined by an ideal $I \subseteq A$. Then
    \begin{equation*}
        T_{[Z]} \Hilb_\rho^G(W) \cong \Hom_A^G(I,A/I),
    \end{equation*}
    where the former denotes the Zariski tangent space at $[Z]$ and the latter denotes the vector space of $A$-linear $G$-equivariant homomorphisms $I \rightarrow A/I$.
\end{proposition}

Let $\mathrm{Univ}_\rho^G(W) \subseteq \Hilb_\rho^G(W) \times_\CC W$ be the universal family. If $H$ is an algebraic group acting on $W$ by $G$-automorphisms, i.e., such that the actions of $G$ and $H$ on $W$ commute, then $H$ acts on both $\mathrm{Univ}_\rho^G(W)$ and on $\Hilb_\rho^G(W)$ such that the projection $\mathrm{Univ}_\rho^G(W) \rightarrow \Hilb_\rho^G(W)$ is $H$-equivariant \cite[Proposition~3.10]{BrionInvariantHilb}.

Denote by $\rho^G$ the maximal trivial subrepresentation of $\rho$ and by $W//G = \Spec(A^G)$ the categorical quotient, where $A^G$ is the subring of $G$-invariants of $A$.

\begin{proposition}[{\cite[Proposition~3.12]{BrionInvariantHilb}}]\label{prop:BrionHilbertChow}
    There is a projective morphism
    \begin{equation*}
        \gamma \colon \Hilb_\rho^G(W) \longrightarrow \Hilb_{\dim(\rho^G)}(W//G),
    \end{equation*}
    called the \emph{Hilbert--Chow morphism}, sending a $G$-stable ideal $I \subseteq A$ to $I \cap A^G$.
\end{proposition}

In the special case $\dim(\rho^G) = 1$, naturally $\Hilb_{1}(W//G) = W//G$.

\subsection{The torus action, reduction to fixed points}
From now on, we let $G = S_n$ and $W = \CC^n$ with the action given by $\sigma \cdot (a_1,\ldots,a_n) = (a_{\sigma^{-1}(1)},\ldots,a_{\sigma^{-1}(n)})$. Let $\mathbb{G}_a = (\CC,+)$ denote the additive group. Then $\mathbb{G}_a^n$ acts on $\Hilb_r(\CC^n)$ by translating the support of an ideal. Clearly, the actions of both $\mathrm{GL}_n = \mathrm{GL}_n(\CC)$ and $\mathbb{G}_a^n$ on $\Hilb_{\dim(\rho)}(\CC^n)$ do not preserve $S_n$-stable ideals. What is left of these actions on $\Hilb^{S_n}_\rho(\CC^n)$ are the centralizers of $S_n$ in $\mathrm{GL}_n$ and $\mathbb{G}_a^n$.
In the case of $\mathbb{G}_a^n$ the centralizer is clearly $\mathbb{G}_a \hookrightarrow \mathbb{G}_a^n$, embedded via the diagonal; in other words, translating a symmetric ideal along a diagonal vector $(a,a,\ldots,a) \in \CC^n$ does not break the symmetry.
For $\mathrm{GL}_n$, the answer is the following.
\begin{lemma}
    Let $n \geq 2$. The centralizer of $S_n$ in $\mathrm{GL}_n$ is a $2$-dimensional torus
    \begin{equation*}
        \mathrm{GL}_n^{S_n} = \left\{ a\mathbbm{1}_n + b\mathbf{1}_n: a,b \in \CC \right\} \cap \mathrm{GL}_n,
    \end{equation*}
    where $\mathbbm{1}_n$ is the identity matrix and $\mathbf{1}_n$ is the matrix filled with ones only. An explicit isomorphism is given by
    \begin{equation*}
        (\CC^\ast)^2 \overset{\cong}{\longrightarrow} \mathrm{GL}_n^{S_n}, \ (s,t) \mapsto t \mathbbm{1}_n + \frac{(s-1)t}{n} \mathbf{1}_n.
    \end{equation*}
\end{lemma}
For the remainder of the paper, we will write $T \coloneqq (\CC^\ast)^2$ for the $2$-dimensional torus and refer to the induced action on $\Hilb^{S_n}_\rho(\CC^n)$ simply as \emph{the $T$-action}. In most situations that we will encounter it suffices to consider only the action of one of the two copies of $\CC^\ast$ embedded into $T$ via $s=1$ or $t=1$.

One of the most useful facts about the $T$-action is that every orbit closure contains a $T$-fixed point \cite[Remarks~3.13(iii)]{BrionInvariantHilb}. In particular, smoothness of $\Hilb^{S_n}_\rho(\CC^n)$ is equivalent to smoothness at every $T$-fixed point because the singular locus is closed. The same is true for every closed geometric property, for example non-reducedness.

\subsection{The associated graded ideal}\label{subsec:assocGraded}
From now on, we abbreviate $\HH_\rho \coloneqq \Hilb^{S_n}_\rho(\CC^n)$.
Let $\CC^\ast$ act on $\CC^n$ by the usual scalar multiplication, inducing the standard grading on $\CC[x_1,\ldots,x_n]$. For a Hilbert function $h$, let $H_{\rho,h}$ denote the locally closed subset of $\HH_\rho$ consisting of all ideals whose associated graded ideal has Hilbert function $h$. We endow $H_{\rho,h}$ with the reduced scheme structure. Then there is a natural map
\begin{equation*}
H_{\rho,h} \longrightarrow \Hilb^{S_n \times \CC^\ast}_{\rho,h}(\CC^n), \ I \mapsto \gr I.    
\end{equation*}
Here, $\gr I$ denotes the associated graded ideal with respect to the increasing filtration by degree and the target Hilbert scheme denotes the (finite) disjoint union of the $(S_n \times \CC^\ast)$-invariant Hilbert schemes running over all $(S_n \times \CC^\ast)$-representations restricting to $\rho$ as an ungraded $S_n$-representation and to $h$ as a $\CC^\ast$-representation (i.e., having Hilbert function $h$).

Formally, this map can be constructed as follows (this is inspired by \cite[Section~4.1]{Cartwright2009HilbertSchemeOf8Points}). Let $\mathcal{Z}_\rho \subseteq \HH_\rho \times \CC^n \rightarrow \HH_\rho$ be the universal family. Base-changing by the embedding $H_{\rho,h} \hookrightarrow \HH_\rho$, we obtain a closed subscheme $\mathcal{Z}_{\rho,h} \subseteq H_{\rho,h} \times \CC^n$, defined by a quasi-coherent ideal sheaf $\mathcal{I} \subseteq \mathcal{O}_{H_{\rho,h}}[x_1,\ldots,x_n]$, and the composition with the projection gives a finite flat family $\mathcal{Z}_{\rho,h} \rightarrow H_{\rho,h}$. Then the associated graded ideal $\gr \mathcal{I} \subseteq \mathcal{O}_{H_{\rho,h}}[x_1,\ldots,x_n]$ defines a closed subscheme $\Tilde{\mathcal{Z}}_{\rho,h} \subseteq H_{\rho,h} \times \CC^n$, and the finite flat family $\Tilde{\mathcal{Z}}_{\rho,h} \rightarrow H_{\rho,h}$ provides the desired map.

\section{$S_n$-Orbits and Tanisaki Ideals}\label{section:Tanisaki}

\subsection{Radical ideals} A first natural question is which representations $\rho$ admit a smooth subscheme, i.e., when does $\mathcal{H}_\rho$ contain a radical ideal? This is essentially answered by the Chinese remainder theorem.

\begin{lemma}\label{lemma:radicalIdeals}
    Let $n \geq 2$ and let $\rho$ be any $S_n$-module. The invariant Hilbert scheme $\mathcal{H}_\rho$ contains a radical ideal if and only if $\rho$ is a direct sum of permutation modules, i.e.,
    \begin{equation*}
        \rho \cong M^{\lambda^1} \oplus \cdots \oplus M^{\lambda^N},
    \end{equation*}
    for some partitions $\lambda^1, \ldots, \lambda^N$ of $n$.
\end{lemma}

\begin{proof}
    Let $I \subseteq P = \CC[x_1,\ldots,x_n]$ be a radical symmetric ideal such that $P/I \cong \rho$ as $S_n$-modules. The vanishing set of $I$ is a symmetric point configuration, consisting of some number $N$ of disjoint $S_n$-orbits with vanishing ideals $I_1, \ldots, I_N$. The Chinese remainder theorem provides a canonical algebra isomorphism $P/I \cong \bigoplus_{i=1}^N P/I_i$ which is $S_n$-equivariant, so it suffices to prove the case of only one orbit. Let $q_1, \ldots, q_k \in \CC^n$ be the distinct points of this orbit and let $\lambda = (\lambda_1 \geq \lambda_2 \geq \cdots \geq \lambda_m \geq 1) \vdash n$ be the corresponding partition, i.e., each $q_i$ has $m$ distinct coordinates which occur with multiplicities $\lambda_1, \ldots, \lambda_m$. Again by the Chinese remainder theorem we obtain a canonical isomorphism $P/I \cong \bigoplus_{i=1}^k P/\mm_{q_i}$ of algebras, in particular of vector spaces. We define an $S_n$-module structure on the direct sum by
    \begin{equation*}
        \overline{f} \in P/\mm_{q_i} \quad \Rightarrow \quad \sigma(\overline{f}) \coloneqq \overline{\sigma(f)} \in P/\mm_{\sigma(q_i)}.
    \end{equation*}
    In this way, the Chinese remainder isomorphism is also $S_n$-equivariant. The canonical inclusion $\CC \subseteq P/\mm_{q_i}$ is an isomorphism, so $\bigoplus_{i=1}^k P/\mm_{q_i} \cong M^\lambda$ as $S_n$-modules.
\end{proof}

\begin{definition}
    Let $\rho$ be a direct sum of permutation modules. The closure of the locus of radical ideals in $\HH_\rho$ is called the \emph{smoothable component}.
\end{definition}

The smoothable component is indeed an irreducible component. The reason is that every radical ideal in the usual Hilbert scheme of points $\Hilb_{\dim \rho}(\CC^n)$ has an open neighborhood containing only radical ideals, so the same is true for the closed subscheme $\Hilb_\rho^{S_n}(\CC^n)$. Moreover, every radical ideal is a smooth point of $\HH_\rho$ because the same is true for the usual Hilbert scheme of points, and taking invariants preserves smoothness by~\cite[Theorem~5.2]{Fogarty1973Fixed}.

\begin{remark}
    Every symmetric \emph{monomial} ideal lies in the smoothable component of some $\HH_\rho$ where $\rho$ is a direct sum of permutation modules. The reason is that the well-known procedure of \emph{distraction}, which is used to prove that any monomial ideal in the usual Hilbert scheme of points lies in the smoothable component, leaves the symmetry intact. This can be seen from the proof of \cite[Lemma~18.10]{Miller2005CombCommAlg}, see also Remark~18.13 in \emph{loc. cit.} Distraction proves in particular that every irreducible component of the usual Hilbert scheme of points intersects the smoothable component. We do not know whether this is also true for $\HH_\rho$ whenever $\rho$ is a direct sum of permutation modules. Let us note that a symmetric monomial ideal is rarely a torus-fixed point of $\HH_\rho$. In fact, the only symmetric monomial ideals which are also torus-fixed points are the powers of the homogeneous maximal ideal.
\end{remark}

\begin{theorem}\label{thm:singleOrbit}
    Let $\lambda \vdash n$ be a partition with $m$ non-zero parts. The invariant Hilbert scheme $\HH_{M^\lambda}$ is connected and has dimension $m$. Moreover, the Hilbert--Chow morphism $\gamma \colon \HH_{M^\lambda} \longrightarrow \CC^n/S_n$ is finite and all fibers over points in $\overline{O(\lambda)}/S_n$ are singletons. Finally,
    \begin{equation*}
        \overline{O(\lambda)}/S_n \subseteq \mathrm{im}(\gamma) \subseteq \bigcup_{\mu \trianglerighteq \lambda} O(\mu)/S_n.
    \end{equation*}
\end{theorem}

The proof of Theorem~\ref{thm:singleOrbit} will be given in the next subsection. As a consequence, $\HH_{M^\lambda}$ is irreducible if and only if so is its image under the Hilbert--Chow morphism if and only if the first inclusion is an equality.

\begin{corollary}\label{corollary:hookPartition}
    Let $\lambda \vdash n$ be a hook partition, i.e., $\lambda = (k,1,1,\ldots,1)$. Then $\HH_{M^\lambda}$ is irreducible.
\end{corollary}

\begin{proof}
    Since $\lambda$ is a hook partition, $\overline{O(\lambda)}/S_n = \bigcup_{\mu \trianglerighteq \lambda} O(\mu)/S_n$. Therefore, $\operatorname{im}(\gamma) = \overline{O(\lambda)}/S_n$, and since the fibers of $\gamma$ over the latter are all singletons, $\gamma$ is a homeomorphism onto its irreducible image. Thus, $\HH_{M^\lambda}$ is irreducible.
\end{proof}

Consider the linear map $\CC^m \hookrightarrow \CC^n$ which sends $(a_1,\ldots,a_m)$ to the $n$-tuple of which the first $\lambda_1$ entries equal $a_1$, the next $\lambda_2$ entries equal $a_2$, and so on. The image of the composition $\varphi_\lambda \colon \CC^m \hookrightarrow \CC^n \twoheadrightarrow \CC^n/S_n$ is $\overline{O(\lambda)}/S_n$, and $\varphi_\lambda \colon \CC^m \rightarrow \overline{O(\lambda)}/S_n$ is a finite morphism. Moreover, $\varphi_\lambda$ factors uniquely as
\begin{equation*}
    \begin{tikzcd}
        \CC^m \ar[r,"\varphi_\lambda"] \ar[d,two heads]
            & \overline{O(\lambda)}/S_n \\
        \CC^m/S_{m_1} \times \cdots \times S_{m_k}, \ar[ur,dotted]
    \end{tikzcd}
\end{equation*}
where $k = |\{\lambda_1, \ldots, \lambda_m\}|$ is the number of distinct parts of $\lambda$ and $m_i$ is the multiplicity of $\lambda_i$ in $\lambda$, so $\sum_{i=1}^k m_i = m$. The quotient $\CC^m/S_{m_1} \times \cdots \times S_{m_k}$ is abstractly isomorphic to $\CC^m$. Moreover, the dotted arrow is a bijection on closed points, and it is a finite morphism since so is $\varphi_\lambda$. Therefore, the dotted arrow is a finite, bijective, birational morphism from a smooth variety and therefore the normalization of $\overline{O(\lambda)}/S_n$.

Let $Z_\lambda$ be the smoothable component of $\HH_{M^\lambda}$, endowed with the reduced scheme structure. By Theorem~\ref{thm:singleOrbit}, restricting the Hilbert--Chow morphism to $Z_\lambda$, we obtain a finite, bijective, birational map $Z_\lambda \longrightarrow \overline{O(\lambda)}/S_n$. Composing it with the normalization $(Z_\lambda)^\nu$ of $Z_\lambda$, we obtain the normalization of $\overline{O(\lambda)}/S_n$ above. In particular,
\begin{equation*}
    (Z_\lambda)^\nu \cong \CC^m,
\end{equation*}
and the normalization map $(Z_\lambda)^\nu \rightarrow Z_\lambda$ is a bijection.

However, it is not unreasonable that something much stronger holds.

\begin{question}\label{question:singleOrbit}
    Let $\lambda \vdash n$ be a partition with $m$ non-zero parts. Is $\HH_{M^\lambda} \cong \CC^m$?
\end{question}

Note that Question~\ref{question:singleOrbit} has an affirmative answer if and only if $\HH_{M^\lambda}$ is smooth. This, in turn, can be checked at a single point of $\HH_{M^\lambda}$, the \emph{Tanisaki ideal}, as we shall soon see.

Evidence for a positive answer to Question~\ref{question:singleOrbit} is provided by Theorem~\ref{thm:singleOrbit}, Corollary~\ref{corollary:hookPartition} and Proposition~\ref{prop:m=2} below, the latter proving the case $m=2$. The answer is positive also in the cases $m=1$ and $m=n$ where the proof is much simpler. In fact, for $m=1$ and $m=n$ the Tanisaki ideals are, respectively, the homogeneous maximal ideal $\mm$ and the complete intersection $(p_1,p_2,\ldots,p_n)$.

One more piece of evidence is provided by Proposition~\ref{prop:reducedPoint}, showing that the $\CC^\ast$-fixed point subscheme of $\HH_{M^\lambda}$, parametrizing the \emph{homogeneous} symmetric ideals $I$ with $P/I \cong M^\lambda$, is smooth, namely a reduced point. This would also be a direct consequence of a positive answer to Question~\ref{question:singleOrbit} by the fact that taking invariants preserves smoothness \cite[Theorem~5.2]{Fogarty1973Fixed}.

\subsection{Tanisaki ideals} We start with the following lemma which can be found in~\cite[Theorem~10.2]{nino2019algorithmic}. We include a proof since it is of central importance for Theorem~\ref{thm:singleOrbit}. 

\begin{lemma}\label{lem:higherSpecht}
    Let $F_T^S$ be a higher Specht polynomial with $T$ any Young tableau of shape $\lambda \vdash n$ and $S$ a standard Young tableau of the same shape. Then $F_T^S$ is divisible by the usual Specht polynomial $\spe_T$. In particular, the Specht ideal of type~$\lambda$ contains the entire $\lambda$-isotypic component of the polynomial ring $P$.
\end{lemma}

\begin{proof}
    It suffices to prove that $F_T^S$ is divisble by $(x_{k_1} - x_{k_2})$ for any two distinct indices $k_1$, $k_2$ appearing in the same column of $T$. For this, in turn, it is enough to see that $\ev(F_T^S) = 0$, where $\ev$ means setting $x_{k_1}$ and $x_{k_2}$ equal. Let $\alpha \in S_n$ be the transposition that exchanges $k_1$ and $k_2$. Then clearly $\ev(F_T^S) = \ev(\alpha(F_T^S))$. However, by definition,
    \begin{align*}
        \alpha(F_T^S) &= \left( \sum_{\sigma \in R(T)} \sum_{\tau \in C(T)} \sgn(\tau) (\alpha \circ \tau) \circ \sigma \right) (x_T^{i(S)}) \\ &= - \left( \sum_{\sigma \in R(T)} \sum_{\tau \in C(T)} \sgn(\alpha \circ \tau) (\alpha \circ \tau) \circ \sigma \right) (x_T^{i(S)}) \\ &= - \left( \sum_{\sigma \in R(T)} \sum_{\tau \in C(T)} \sgn(\tau) \tau \circ \sigma \right) (x_T^{i(S)}) = - F_T^S,
    \end{align*}
    so $\ev(F_T^S) = \ev(\alpha(F_T^S)) = - \ev(F_T^S)$, hence $\ev(F_T^S) = 0$.
\end{proof}

In fact, the proof shows that there is nothing special here about $x_T^{i(S)}$. Instead, applying the Young symmetrizer $\sum_{\sigma \in R(T)} \sum_{\tau \in C(T)} \sgn(\tau) \tau \circ \sigma$ to \emph{any} polynomial results in a multiple of $\spe_T$ (possibly zero, of course). 
By the theory of higher Specht polynomials \cite{Ariki1997Higher,Terasoma1993Higher}, the lowest degree in which the Specht module $S^\lambda$ occurs in the polynomial ring $P$ is $d(\lambda) \coloneqq \sum_{i=1}^n (i-1)\lambda_i$. The multiplicity of $S^\lambda$ in $P_{d(\lambda)}$ is~$1$, and it is generated by the Specht polynomials of shape~$\lambda$.

Recall next the apolarity pairing

\begin{equation*}
    P \times P \longrightarrow P, \quad \langle f , g \rangle \coloneqq f(\partial_1, \ldots, \partial_n)(g(x_1, \ldots, x_n)).
\end{equation*}

This pairing is $S_n$-equivariant in the sense that $\sigma(\langle f, g \rangle) = \langle \sigma(f), \sigma(g) \rangle$. For every $d$, it restricts to a non-degenerate symmetric bilinear form $P_d \times P_d \longrightarrow \CC$ for which the monomials of degree $d$ form an orthogonal basis. In fact, it is different from the standard scalar product only by a diagonal matrix with positive integers as diagonal entries. In particular, the restriction to $\RR$-coefficients yields an inner product. Let now $\mu_1$ and $\mu_2$ be two distinct partitions of $n$ and denote by $V^{\mu_i}$ the $\mu_i$-isotypic component in $P_d$ ($i = 1,2$). Then $\langle V^{\mu_1}, V^{\mu_2} \rangle = 0$ because the orthogonal complement $\left(V^{\mu_1}\right)^\perp$ is an $S_n$-submodule of $P_d$ which intersects $V^{\mu_1}$ trivially. Considering again the entire apolarity pairing, for any subset $W \subseteq P$, the set
\begin{equation*}
    \{f \in P : \langle f, g \rangle = 0 \text{ for all } g \in W\}
\end{equation*}
is an ideal of $P$, called the \emph{ideal apolar to $W$}, or simply the \emph{apolar ideal}.

Let now $I \subseteq P = \CC[x_1,\ldots,x_n]$ be a homogeneous symmetric ideal such that $P/I \cong M^\lambda$ as $S_n$-modules.
By Lemma~\ref{lem:higherSpecht}, $I$ must intersect $S^\lambda \subseteq P_{d(\lambda)}$ trivially since $P/I \cong M^\lambda$ contains $S^\lambda$ with multiplicity~$1$. It follows that $I_{d(\lambda)}$, the degree~$d(\lambda)$ part of $I$, is contained in the apolar ideal to $S^\lambda \subseteq P_{d(\lambda)}$. We denote this apolar ideal by $I_\lambda$ and call it the \emph{Tanisaki ideal}. We claim $I \subseteq I_\lambda$. This is because otherwise, if there was some homogeneous $f \in I \setminus I_\lambda$, then by definiton of $I_\lambda$ there would be some homogeneous $g \in S^\lambda \subseteq P_{d(\lambda)}$ such that $\langle f, g \rangle \neq 0$. Hence, if $h$ is any monomial appearing in $\langle f, g \rangle \in P$, we get $\langle hf, g \rangle \neq 0$. But $hf \in I_{d(\lambda)} \subseteq (I_\lambda)_{d(\lambda)}$, a contradiction. This shows that the Tanisaki ideal is the unique largest homogeneous symmetric ideal of $P$ intersecting $S^\lambda \subseteq P_{d(\lambda)}$ trivially, as mentioned in \cite[p.~62]{Haiman2003Combinatorics}.

From the inclusion $I \subseteq I_\lambda$ it follows that $\dim_\CC(P/I_\lambda) \leq \dim_\CC(M^\lambda) = \binom{n}{\lambda}$. The final point is now that actually $\dim_\CC(P/I_\lambda) = \binom{n}{\lambda}$ (proved in~\cite{BergeronGarsia1992Harmonic}, recalled in~\cite[Lemma~3.4.17]{Haiman2003Combinatorics}), hence $I = I_\lambda$. This shows that $I_\lambda$ is the only homogeneous ideal in $\mathcal{H}_{M^\lambda}$ which proves at once connectedness of $\mathcal{H}_{M^\lambda}$ and that the Tanisaki ideal is its only torus-fixed point. We can even strengthen this result in two ways.

\begin{proposition}\label{prop:reducedPoint}
    Let $\CC^\ast$ be embedded into the torus $T$ via $s=1$. Then the $\CC^\ast$-fixed point subscheme $\HH_{M^\lambda}^{\CC^\ast}$ is a reduced point.
\end{proposition}

\begin{proof}
    The fixed point subscheme $\HH_{M^\lambda}^{\CC^\ast}$ parametrizes the homogeneous symmetric ideals $I$ with $P/I \cong M^\lambda$ as ungraded $S_n$-representations. The preceeding discussion shows that $\HH_{M^\lambda}^{\CC^\ast}$ set-theoretically consists of a single point, the Tanisaki ideal $I_\lambda$. To see that this point is reduced, it suffices to show that there are no non-trivial (homogeneous and symmetric) infinitesimal deformations of $I_\lambda$. For this, let $A$ be an Artinian local $\CC$-algebra and $I \subseteq P_A \coloneqq P \otimes_\CC A$ be a homogeneous symmetric ideal such that $A \rightarrow P_A/I$ is flat and lifts $P/I_\lambda$. Since $A$ is local and $P_A/I$ is a finitely generated $A$-module, $P_A/I$ is free of rank $\dim M^\lambda$. We also write $I_\lambda \subseteq P_A$ for the extended Tanisaki ideal. If we can show that $I \subseteq I_\lambda$ then the equivariant surjection $P_A/I \twoheadrightarrow P_A/I_{\lambda}$ can be regarded as a surjective endomorphism of a finite rank free $A$-module and is hence an isomorphism, proving $I = I_\lambda$, as desired.

    To conclude, assume there is some homogeneous $f \in I \setminus I_\lambda$. Then $f$ does not annihilate the entire Specht module $S^\lambda \subseteq P_{d(\lambda)}$, so there is $g \in S^\lambda \subseteq P_{d(\lambda)}$ with $\langle f, g \rangle \neq 0$. Let $h$ be an arbitrary monomial appearing in $\langle f, g \rangle$. Then $\langle hf, g \rangle \neq 0$. Replacing $f$ by $hf \in I \setminus I_\lambda$ we are now given $f \in I_{d(\lambda)} \setminus I_\lambda$. The $S_n$-submodule of $I$ generated by all permutations of $f$ must then contain a scalar multiple of the unique $S^\lambda \subseteq P_{d(\lambda)}$, so there is a non-zero $a \in A$ with $a \spe_T \in I$ for some standard Young tableau $T$ of shape $\lambda$. But then, by Lemma~\ref{lem:higherSpecht}, the isotypic component of $S^\lambda$ in $P_A/I$ is $a$-torsion, so if $a$ is not a unit, then $P_A/I$ cannot be a free $A$-module, a contradiction. On the other hand, if $a$ is a unit, again by Lemma~\ref{lem:higherSpecht}, $I$ contains the entire Specht ideal of type $\lambda$, so $P_A/I \otimes_A \CC = P/I_\lambda$ does not contain any copy of $S^\lambda$, contradicting $P/I_\lambda \cong M^\lambda$.
\end{proof}

\begin{remark}\label{rmk:noPositiveTangents}
    Proposition~\ref{prop:reducedPoint} can be slightly improved: The proposition is equivalent to saying that the degree~$0$ part of $T_{[I_\lambda]} \HH_{M^\lambda} \cong \Hom^{S_n}_P(I_\lambda, P/I_\lambda)$ vanishes. We claim that we even have $(T_{[I_\lambda]} \HH_{M^\lambda})_{\geq 0} = 0$, so there are no non-negative tangents. The essential point is again that $S^\lambda$ does not appear in $P_d$ for all degrees $d < d(\lambda)$. Indeed, a tangent vector of degree $k \geq 0$ is given by an $S_n$-equivariant $P$-module homomorphism $\varphi: I_\lambda \rightarrow P/I_\lambda$ of degree $k$, i.e., $\varphi( (I_\lambda)_d ) \subseteq (P/I_\lambda)_{d+k}$ for all $d$. This tangent vector $\varphi$ corresponds to a deformation of $P/I_\lambda$ over the dual numbers $\CC[\epsilon]$ corresponding to the ideal $I_\varphi \subseteq P_{\CC[\epsilon]}$ given by
    \begin{equation*}
        I_\varphi = \{x + \epsilon y | x \in I_\lambda, \ y \in P, \ y = \varphi(x) \text{ in } P/I_\lambda \}.
    \end{equation*}
    In particular, if $y \in P_{d(\lambda)}$, then $x \in (I_\lambda)_{d(\lambda) - k}$. Since $\varphi$ is $S_n$-equivariant and $(I_\lambda)_{\leq d(\lambda)}$ does not contain any copy of $S^\lambda$, we obtain that there is no $x + \epsilon y$ in $I_\varphi$ such that $0 \neq y \in S^\lambda \subseteq P_{d(\lambda)}$. Then a similar proof as that of Proposition~\ref{prop:singleton} below shows that $I_\varphi$ is just the extension of $I_\lambda$ to $P_{\CC[\epsilon]}$, so $\varphi = 0$.
\end{remark}

\begin{proposition}\label{prop:singleton}
    Let $I \subseteq P$ be a not necessarily homogeneous symmetric ideal such that $P/I \cong M^\lambda$ as $S_n$-modules. If $P/I$ is supported only at the origin, i.e., $V(I) = \{0\}$, then $I = I_\lambda$.
\end{proposition}

\begin{proof}
    It is enough to prove that $I \subseteq I_\lambda$. Assume this is not true and there exists $f \in I \setminus I_\lambda$. Write $f = f_1 + f_2 + \cdots$ as the sum of its homogeneous parts. Then there is a smallest degree $d$ such that there exists some $g \in S^\lambda \subseteq P_{d(\lambda)}$ with $\langle f_d, g \rangle \neq 0$. Let $h$ be an arbitrary monomial appearing in $\langle f_d, g \rangle$, then $\langle hf, g \rangle = \langle hf_d, g \rangle \neq 0$. Replacing $f$ by $hf \in I \setminus I_\lambda$ we are now given a possibly inhomogeneous element $f \in I$ such that all homogeneous components of $f$ except for $f_{d(\lambda)}$ annihilate $S^\lambda \subseteq P_{d(\lambda)}$. Since $S^\lambda$ occurs in $P$ only in degrees $\geq d(\lambda)$, the $S_n$-submodule of $I$ generated by all permutations of $f$ in particular contains a polynomial $p = p_{d(\lambda)} + p_{d(\lambda)+1} + \cdots$ with $0 \neq p_{d(\lambda)} \in S^\lambda \subseteq P_{d(\lambda)}$ such that all permutations of $p$ span an $S_n$-module isomorphic to $S^\lambda$, so we may even assume $p_{d(\lambda)}$ to be a Specht polynomial of shape $\lambda$. By Lemma~\ref{lem:higherSpecht}, then, all $p_d$ for $d > d(\lambda)$ are divisible by $p_{d(\lambda)}$. Therefore, we can eliminate all degree $>d(\lambda)$ terms of $p$ with multiples of $p$ itself, creating only new terms in strictly higher degrees. Since $V(I) = \{0\}$, there exists some $N$ with $\mm^N \subseteq I$. Therefore, eventually, $p_{d(\lambda)} \in I$. But then, again according to Lemma~\ref{lem:higherSpecht}, $I$ contains the entire Specht ideal of type $\lambda$, therefore $P/I$ cannot contain any copy of $S^\lambda$, contradicting $P/I \cong M^\lambda$.
\end{proof}

\begin{corollary}\label{corollary:HilbChowFinite}
    The Hilbert--Chow morphism $\gamma \colon \HH_{M^\lambda} \longrightarrow \CC^n/S_n$ is finite.
\end{corollary}

\begin{proof}
    The Hilbert--Chow morphism is projective by Proposition~\ref{prop:BrionHilbertChow}. It is also equivariant with respect to the torus action. Since the Tanisaki ideal $I_\lambda$ is in the closure of the torus-orbit of every point, the upper semicontinuity of fiber dimensions reduces the finiteness of $\gamma$ to the finiteness of the fiber over the origin in $\CC^n/S_n$. Set-theoretically, this fiber is a singleton by Proposition~\ref{prop:singleton}.
\end{proof}

\begin{proposition}\label{prop:HilbChowInjective}
    Let $q \in \overline{O(\lambda)}$. Then the fiber of the Hilbert--Chow morphism $\gamma \colon \HH_{M^\lambda} \longrightarrow \CC^n/S_n$ over the point corresponding to the orbit ${S_n \cdot q}$ is a singleton.
\end{proposition}

For the proof, we need a few facts from the theory of higher Specht polynomials in the case of \emph{products} of symmetric groups $G \coloneqq S_{l_1} \times \cdots \times S_{l_k}$, $\sum_{i=1}^k l_i = n$, acting on the polynomial ring $P$ via $G \subseteq S_n$, see~\cite{Ariki1997Higher}. Fundamentally, the irreducible representations of $G$ are exactly the tensor products of Specht modules \mbox{$S^{\nu^1} \otimes \cdots \otimes S^{\nu^k}$}, $\nu^i \vdash l_i$. This tensor product can be realized concretely as a $G$-submodule of $P_{d(\nu^1) + \ldots + d(\nu^k)}$ spanned by the products of the Specht polynomials for the $\nu^i$. In fact, the smallest degree $d$ such that $P_d$ contains a copy of \mbox{$S^{\nu^1} \otimes \cdots \otimes S^{\nu^k}$} is precisely $d(\nu^1) + \ldots + d(\nu^k)$, and in this degree the multiplicity of \mbox{$S^{\nu^1} \otimes \cdots \otimes S^{\nu^k}$} is~$1$. Denote by $\mathbf{T} \coloneqq (T_1, \ldots, T_k)$ a $k$-tuple of tableaux with $T_i$ of shape $\nu^i$, filled in such a way that the integers $l_1 + \ldots + l_{i-1} + 1, \ldots, l_1 + \ldots + l_i$ all appear exactly once in $T_i$. We write $F_{\mathbf{T}} \coloneqq \prod_{i=1}^k F_{T_i}$. Lemma~\ref{lem:higherSpecht} generalizes to this setting, i.e., every higher Specht polynomial $F_{\mathbf{T}}^{\mathbf{S}}$ (see \cite{Ariki1997Higher} for the notation) is divisible by $F_{\mathbf{T}}$. This can be proven similarly as Lemma~\ref{lem:higherSpecht}.

\begin{proof}[Proof of Proposition~\ref{prop:HilbChowInjective}]
    We extend the proof of Proposition~\ref{prop:singleton}. Let $q \in \overline{O(\lambda)} \subseteq \CC^n$. Up to a permutation, the first $l_1$ entries of $q$ are equal to $a_1 \in \CC$, the following $l_2$ entries are equal to $a_2 \in \CC$, and so on, where $a_1, \ldots, a_k$ are all distinct. Each multiplicity $l_i$ is a sum of some parts of $\lambda$ in such a way that every part of $\lambda$ shows up for exactly one of the $l_i$. In other words, there is a surjective map $l:[m] \rightarrow [k]$ such that $l_i = \sum_{j \in l^{-1}(i)} \lambda_j$. Denote by $\mu^i \vdash l_i$ the partition given by those parts of $\lambda$ indexed by $l^{-1}(i)$. We now denote by $q = q_1, q_2, \ldots, q_{s}$ the distinct points of the $S_n$-orbit of $q$. For every ideal $I \subseteq P$ in the fiber of the Hilbert--Chow morphism over this orbit, the Chinese remainder theorem gives a canonical $S_n$-equivariant isomorphism
    \begin{equation*}
        P/I \cong P/I_{q_1} \oplus \cdots \oplus P/I_{q_{s}}, \quad f \mapsto (f, \ldots, f),
    \end{equation*}
    where the $S_n$-module structure on the right side is defined similarly as in the proof of Proposition~\ref{lemma:radicalIdeals}. For the uniqueness of $I$ it will be enough to prove the uniqueness of the $I_{q_i}$ because $I = \bigcap_{i=1}^{s} I_{q_i}$. Without loss of generality we only show the uniqueness of $I_{q}$.
    
    Let $G = S_{l_1} \times \cdots \times S_{l_k}$ denote the stabilizer of $q$, so that $P/I_q$ is a $G$-module. We define $I' \subseteq P$ to be the apolar ideal to the $G$-module
    \begin{equation*}
        V \coloneqq \left( S^{\mu^1} \otimes 1 \otimes \cdots \otimes 1 \right) \oplus \cdots \oplus \left( 1 \otimes 1 \otimes \cdots \otimes S^{\mu^k} \right) \subseteq P,
    \end{equation*}
    all direct summands embedded in the smallest possible degrees $d(\mu^1), \ldots, d(\mu^k)$, i.e., $V$ is spanned by all Specht polynomials for the partitions $\mu^1, \ldots, \mu^k$ in the respective sets of variables.
    (If some of the $\mu^i = (l_i)$ are trivial, we do not include the corresponding direct summands into $V$.)
    For any $f \in P$ we can write $f = f^1 + \ldots + f^k + g$, where $f^i$ is the sum of all terms of $f$ which only involve the $i$-th set of variables, i.e., those indexed by $l_1 + \ldots + l_{i-1} + 1, \ldots, l_1 + \ldots + l_i$, and $g$ is the sum of all remaining terms. Certainly, differentiating a polynomial with respect to a variable it does not contain gives zero. Therefore, since the $i$-th direct summand of $V$ only involves the $i$-th set of variables, it is clear that $\langle f, V \rangle = 0$ if and only if $\langle f^i, S^{\mu^i} \rangle = 0$ for all $i = 1, \ldots, k$, where $S^{\mu^i}$ is understood again as the span of all Specht polynomials of type $\mu^i$ in the $i$-th set of variables.
    This proves
    \begin{equation*}
        P/I' \cong \left( \CC[x_1,\ldots,x_{l_1}]/I'_1 \right) \otimes \cdots \otimes \left( \CC[x_{n-l_k+1},\ldots, x_n]/I'_k \right),
    \end{equation*}
    where $I'_1$ is the apolar ideal to $S^{\mu^1} \subseteq \CC[x_1,\ldots,x_{l_1}]_{d(\mu^1)}$ etc., i.e., the $I'_i = I_{\mu^i}$ are the Tanisaki ideals corresponding to the partition $\mu^i \vdash l_i$. In particular, $P/I' \cong M^{\mu^1} \otimes \cdots \otimes M^{\mu^k}$ as $G$-modules.
    
    Let now $I'_q$ be the translate of $I'$ whose only support point is $q$; this translation is $G$-equivariant. We claim $I_q \subseteq I'_q$. For this, we first observe that the canonical map $V \rightarrow P/I_q$ is injective, i.e., no direct summand of $V$ is contained in $I_q$. The reason is that the restriction of the Chinese remainder isomorphism above to $S^\lambda \subseteq P_{d(\lambda)}$ is injective. Now, $S^\lambda \subseteq P_{d(\lambda)}$ is generated by the Specht polynomials of shape $\lambda$, each of which is a multiple of the product of the corresponding Specht polynomials of the $\mu^i$. In particular, any Specht polynomial for $\mu^i$ is non-zero in $P/I_{q_i}$, hence indeed $V$ injects into $P/I_q$.
    To see that $I_q \subseteq I'_q$, we can first translate both ideals $G$-equivariantly into the origin. Note that $V$ still injects into $P/I_0$ because the Specht polynomials generating $V$ are not affected by this translation. Now assume that there exists $f \in I_0 \setminus I'$. Writing $f = f^1 + \ldots + f^k + g$ as above, we can assume without loss of generality that $\langle f^1, S^{\mu^1} \rangle \neq 0$. Multiplying $f$ by some monomial $h$ appearing in $\langle f^1, S^{\mu^1} \rangle$, we may replace $f$ by $hf$, so that we can assume that the degree $d(\mu^1)$ part of $f^1$ does not annihilate $S^{\mu^1}$. Hence the $G$-submodule of $I_0$ generated by $f$ contains an element $f'$ which spans a $G$-submodule isomorphic to $S^{\mu^1} \otimes 1 \otimes \cdots \otimes 1$ such that the degree $d(\mu^1)$ part of $f'$ is a Specht polynomial in $S^{\mu^1}$. Moreover, $f'$ does not have any non-zero homogeneous components of smaller degrees because $P_{<d(\mu^1)}$ does not contain any copy of $S^{\mu^1} \otimes 1 \otimes \cdots \otimes 1$. Therefore, we may assume $f' = a + b$ with $a \in S^{\mu^1} \subseteq P_{d(\mu^1)}$ a Specht polynomial and $b \in P_{>d(\mu^1)}$. As $f'$ spans a $G$-module isomorphic to $S^{\mu^1} \otimes 1 \otimes \cdots \otimes 1$, the polynomial $b$ lies in the corresponding isotypic component of $P$ and is divisible by $a$ by Lemma~\ref{lem:higherSpecht}. Thus inductively, since $\mm^N \subseteq I_0$ for some $N$, we obtain $a \in I_0$. Therefore, the canonical map $V \rightarrow P/I_0$ has a non-trivial kernel, contradicting what we said above. So indeed, $I_q \subseteq I'_q$. This implies
    \begin{align*}
        \binom{l_1}{\mu^1} \cdots \binom{l_k}{\mu^k} = \dim_\CC(P/I'_q) \leq \dim_\CC(P/I_q) &= \frac{\dim_\CC(P/I)}{r} \\ &= \frac{\binom{n}{\lambda}}{|S_n|/|G|} = \frac{\binom{n}{\lambda}}{\binom{n}{l_1,\ldots,l_k}} = \binom{l_1}{\mu^1} \cdots \binom{l_k}{\mu^k},
    \end{align*}
    so for dimension reasons actually $I_q = I'_q$. This proves the uniqueness of~$I_q$.
\end{proof}

\begin{proof}[Proof of Theorem~\ref{thm:singleOrbit}]
    Connectedness follows from the fact that the Tanisaki ideal $I_\lambda$ is the only torus-fixed point. Since the smoothable component has dimension $m$, it suffices to show that every other potential irreducible component has dimension $\leq m$. Since the Hilbert--Chow morphism is finite by Corollary~\ref{corollary:HilbChowFinite}, the dimension of every irreducible component agrees with the dimension of its image. So it remains to prove that all partitions $\mu$ for which the preimage of $O(\mu)/S_n$ under the Hilbert--Chow morphism is non-empty have at most $m$ parts. For this, let $I \subseteq P$ be a symmetric ideal with $P/I \cong M^\lambda$ as $S_n$-modules and such that the $S_n$-orbit given by $V(I)$ lies in $O(\mu)/S_n$. Then $P/\sqrt{I} \cong M^\mu$ and the canonical surjection $P/I \twoheadrightarrow P/\sqrt{I}$ is $S_n$-equivariant. Therefore, $\mu \trianglerighteq \lambda$. This in particular implies that the number of non-zero parts of $\mu$ is at most $m$, as desired.
    The remaining statement is precisely Proposition~\ref{prop:HilbChowInjective}.
\end{proof}

Remarkably, it is possible to give explicit generators for the Tanisaki ideal $I_\lambda$, as follows.

\begin{theorem}[\cite{Tanisaki1982Defining,GarsiaProcesi1992}]
    Let $\lambda \vdash n$ and let $\lambda'$ be its transpose. For $S \subseteq [n]$ we write $e_r(x_S)$ for the elementary symmetric polynomial of degree $r$ in the variables indexed by $S$. Then the Tanisaki ideal $I_\lambda$ is generated by the polynomials $e_r(x_S)$ with $|S| \geq r \geq r_\lambda(|S|) \coloneqq |S| - n + 1 + \sum_{i=1}^{n-|S|} \lambda'_i$.
\end{theorem}

A short computation shows that $|S| \geq r_\lambda(|S|)$ is equivalent to $|S| \geq n - \lambda_1 + 1$. Under this assumption, if $S \subseteq S'$, then $r_\lambda(|S|) \geq r_\lambda(|S'|)$. This together with the relation $e_r(x_{S \setminus i}) = e_r(x_S) - x_i e_{r-1}(x_{S \setminus i})$ for $i \in S$ shows that $I_\lambda$ is in fact generated already by the full elementary symmetric polynomials $e_r(x_{[n]})$ for $1 \leq r \leq n$ and by the $e_{r_\lambda(|S|)}(x_S)$. In fact, among the full elementary symmetric polynomials, it is enough to require $e_r(x_{[n]})$ to lie in the ideal for $1 \leq r \leq m$, where $m$ is the number of non-zero parts of $\lambda$.

\begin{proposition}\label{prop:m=2}
    Let $\lambda = (\lambda_1,\lambda_2) \vdash n$ be a partition with only $2$ non-zero parts. Then $\HH_{M^\lambda} \cong \CC^2$.
\end{proposition}

\begin{proof}
It suffices to show that the Tanisaki ideal $I_\lambda$ is a smooth point of $\HH_{M^\lambda}$ since it is the only torus-fixed point. The tangent space at $[I_\lambda]$ is isomorphic to $\Hom_{P}^{S_n}(I_\lambda,P/I_\lambda)$, the vector space of all $S_n$-equivariant $P$-module homomorphisms $I_\lambda \rightarrow P/I_\lambda$. Since $I_\lambda$ lies in the smoothable component of $\HH_{M^\lambda}$ which has dimension~$2$, it is enough to show that the dimension of the tangent space is at most~$2$.
		
	For this, note first that $r_\lambda(n-l) = l+1$ for all $0 \leq l \leq \lambda_2$ and that $\lambda_2 + 1 = r_\lambda(n-\lambda_2) = r_\lambda(n-\lambda_2 - 1) = \cdots = r_\lambda(n-\lambda_1 + 1)$. We abbreviate $e_r = e_r(x_{[n]})$. We will now repeatedly use the relation $e_r(x_{S \setminus i}) = e_r(x_S) - x_i e_{r-1}(x_{S \setminus i})$ for all $i \in S$. First we can observe that
	\begin{equation*}
		e_2(x_{[n-1]}) = e_2 - x_n e_1(x_{[n-1]}) = e_2 - x_n e_1 + x_n^2,
	\end{equation*}
	hence $x_n^2 \in I_\lambda$. Let $I \coloneqq (p_1,x_1^2, \ldots, x_n^2) \subseteq I_\lambda$, where $p_1 = e_1$ is the first power sum polynomial. Using the Newton identities and the last equation we can see that $e_2(x_{[n-1]}) \in I$. Similarly, $e_r(x_{[n-1]}) \in I$ for all $r \geq 2 = r_\lambda(n-1)$. Inductively, using the same idea, for all $l \leq \lambda_2$ we have
	\begin{align*}
		e_r(x_{[n-l]}) &= e_r(x_{[n-l]\cup n}) - x_n e_{r-1}(x_{[n-l]}) \\ &= e_r(x_{[n-l]\cup n}) - x_n e_{r-1}(x_{[n-l]\cup n}) + x_n^2 e_{r-2}(x_{[n-l]}),
	\end{align*}
	and for $r \geq r_\lambda(n-l) = l+1 = r_\lambda(n-l+1) + 1$ all three terms lie in $I$. If $\lambda_1 \leq \lambda_2 + 1$, this implies $I_\lambda = I$. If $\lambda_1 \geq \lambda_2 + 2$, we continue as follows: For all $l$ in the range $\lambda_1 - 1 \geq l \geq \lambda_2 + 1$, the partial elementary symmetric polynomials $e_{r_\lambda(|S|)}(x_S)$, $|S| = n-l$, all have the same degree $\lambda_2 + 1$, and for $|S| = n-\lambda_1+1 = \lambda_2 + 1$ we get the square-free monomial of degree $\lambda_2 + 1$, so all other generators in this range are redundant. In the end we obtain $I_\lambda = (p_1,x_1^2, \ldots, x_n^2, S_n \cdot (x_1 x_2 \cdots x_{\lambda_2+1}))$, where $S_n \cdot (x_1 x_2 \cdots x_{\lambda_2+1}) $ denotes the $S_n$-orbit of the monomial $x_1 x_2 \cdots x_{\lambda_2+1}$. On the other hand, $P/I_\lambda \cong M^\lambda = \oplus_{0 \leq k \leq \lambda_2} S^{(n-k,k)}$ as follows from the combinatorial formula for the Kostka numbers. These two facts together give even more: If $\lambda_1 \geq \lambda_2 + 2$, what we have proven so far applied to the partition $\mu \coloneqq (\lambda_1 - 1, \lambda_2 + 1)$ shows further that $I_\lambda$ is already generated by $I$ together with only the \emph{Specht polynomials} of shape $\mu = (\lambda_1 - 1,\lambda_2 + 1)$. Indeed, the canonical surjection
	\begin{equation*}
		\oplus_{0 \leq k \leq \lambda_2+1} S^{(n-k,k)} \cong P/I_\mu \twoheadrightarrow P/I_\lambda \cong \oplus_{0 \leq k \leq \lambda_2} S^{(n-k,k)}
	\end{equation*}
	has kernel $S^\mu$, embedded in degree $d(\mu) = \lambda_2 + 1$. Moreover, each Specht module $S^{(n-k,k)}$ in $P/I_\lambda$ is embedded in degree $k$, hence in degree $< \lambda_2 + 1$. Therefore, if $I'$ denotes the ideal generated by $I$ and all Specht polynomials of shape $\mu$, we have a canonical surjection $P/I' \twoheadrightarrow P/I_\lambda$ which is an isomorphism in degrees $\leq \lambda_2 + 1$. But $I_\lambda$ is generated in degrees $\leq \lambda_2 + 1$, so $I' = I_\lambda$.
		
	Since $M^\lambda$ does not contain any copy of $S^\mu$, we deduce that a homomorphism $f \in \Hom_{P}^{S_n}(I_\lambda,P/I_\lambda)$ is determined already by its values $f(p_1) = \alpha \cdot 1$, $f(p_2) = \beta \cdot 1$ and $f(x_i^2-x_j^2) = \gamma \cdot (x_i - x_j)$ for some $\alpha, \beta, \gamma \in \CC$. Hence it suffices to show that $\alpha$, $\beta$, $\gamma$ must satisfy some non-trivial linear relation. If $\lambda_1 > \lambda_2$ we will show that necessarily $\beta = 0$ and if $\lambda_1 = \lambda_2 = \frac{n}{2}$ we will show that $\alpha = \frac{n}{2} \gamma$, finishing the proof. We start with $\lambda_1 > \lambda_2$. Since $x_1 x_2 \cdots x_{\lambda_2+1} \in I_\lambda$, we can consider its image under $f$. Write $f(x_1 x_2 \cdots x_{\lambda_2+1}) = f_0 + f_1 + \ldots + f_{\lambda_2}$ as a sum of its graded pieces. Then $x_1^2 x_2 \cdots x_{\lambda_2 + 1} \in I_\lambda$ is both a multiple of $x_1^2$ and of $x_1 x_2 \cdots x_{\lambda_2+1}$ and therefore maps to both $\frac{\beta}{n} x_2 \cdots x_{\lambda_2 + 1}$ and $x_1 f(x_1 x_2 \cdots x_{\lambda_2+1})$ at the same time. Therefore, all graded pieces of the two must be equal in $P/I_\lambda$, in particular
	\begin{equation*}
		\beta x_2 \cdots x_{\lambda_2 + 1} - n x_1 f_{\lambda_2 - 1} \in I_\lambda.
	\end{equation*}
	Hence, this polynomial must annihilate every Specht polynomial of shape $\lambda$. However, since $\lambda_1 > \lambda_2$, there exists a Specht polynomial not involving the variable $x_1$ at all (putting the $1$ in the upper right corner of the tableau) and containing the monomial $x_2 \cdots x_{\lambda_2 + 1}$ (putting these indices successively in the second row, for example). This Specht polynomial is annihilated by the above polynomial if and only if $\beta = 0$.
		
	If $\lambda_1 = \lambda_2 = \frac{n}{2}$, we have the following relation: Define $g \coloneqq (x_1-x_2) (x_3-x_4) \cdots (x_{n-1} - x_n)$. Then
	\begin{equation*}
		g p_1 = \sum_{i \text{ odd}} \frac{g}{x_i - x_{i+1}} (x_i^2 - x_{i+1}^2).
	\end{equation*}
	This implies that $(\alpha - \frac{n}{2}\gamma) \cdot g \in I_\lambda$ which, since $g$ is a Specht polynomial of shape $\lambda$, is only possible if $\alpha = \frac{n}{2} \gamma$.
\end{proof}

\subsection{Two ideal inclusions}
While exploring Tanisaki ideals we encountered the containments of a certain ideal generated by Specht polynomials, a simple monomial ideal and the Tanisaki ideal. Although these inclusions are not used in our understanding of the invariant Hilbert schemes, we do include them as we have not found them in the literature.

For a partition $\mu \vdash n$ of length $m$ and $1 \leq k \leq m-1$ we define $R_k(\mu):=\mu_m+\mu_{m-1}+\ldots+\mu_{k+1}+1$, and $R_m(\mu) := 1$. We define an ideal $\Tilde{I}_\mu$ associated with $\mu$ as $$\Tilde{I}_\mu = (p_1,\ldots,p_{m-1},S_n \cdot x_1^m,S_n \cdot (x_1x_2\cdots x_{R_k(\mu)})^k : 1 \leq k \leq m-1).$$ 

We claim $\Tilde{I}_\mu \subseteq I_\mu$ and $(p_1,\ldots,p_{m-1},\spe_T | \sh (T) \not \unrhd \mu) \subseteq I_\mu$. The Tanisaki ideal of $\mu$ is the associated graded ideal to the $S_n$-orbit of any point $a \in \CC^n$ whose stabilizer in $S_n$ is conjugate to the Young subgroup $S_\mu$ \cite{GarsiaProcesi1992}. To see the first inclusion let $a_1,\ldots,a_m \in \CC$ be the distinct coordinates of $a$ such that $a_i$ occurs $\mu_i$ times in $a$. Then for any $1 \leq k \leq m$ and any subset $I\subset [n]$ of cardinality $R_k(\mu)$ we have $\{a_i | i \in I\} \cap \{a_1\ldots,a_k\} \neq \emptyset$ by the pigeonhole principle. Thus the polynomial $\prod_{i=1}^{R_k(\mu)}\prod_{j=1}^k(x_i-a_j)$ vanishes on the entire orbit $S_n \cdot \{a\}$. Also $\prod_{i=1}^m(x_1-a_i)$ vanishes on the orbit. However, the top degree terms of those polynomials are precisely the monomial generators of the ideal $\Tilde{I}_\mu$. This proves $\Tilde{I}_\mu \subset I_\mu$.
For the other inclusion, we note that for $a \in \CC^n$ we know that $p_i-p_i(a)$ is contained in the vanishing ideal of $S_n \cdot \{a\}$ and the containment of $\spe_T$ with $\sh(T) \not \unrhd \mu$ follows from \cite[Corollary~1]{moustrou2021symmetric}.

\begin{lemma}
    We have $(p_1,\ldots,p_n,\spe_T : \sh (T) \not \unrhd \mu) \subseteq \Tilde{I}_\mu \subseteq I_\mu$. 
\end{lemma}

\begin{proof}
It remains to prove the first inclusion. For this, we show the stronger containment
\[(\spe_T | \sh (T) \not \unrhd \mu) \subseteq (S_n \cdot x_1^m, S_n \cdot (x_1 x_2 \cdots x_{R_k(\mu)})^k : 1 \leq k \leq m-1) .\]
Let $T$ be a Young tableau of shape $\lambda \not \unrhd \mu$ with $l$ columns $C_1,\ldots,C_l$. We have $$\spe_T = \pm \prod_{k=1}^l\prod_{\substack{i < j \\ i,j \in C_k}} (x_i-x_j).$$ Any monomial occurring in $\spe_T$ is a permutation of the monomial
$$x_1^{|C_1|-1}x_2^{|C_1|-2}\cdots x^1_{|C_1|-1} \cdot x_{|C_1|}^{|C_2|-1} x_{|C_1|+1}^{|C_2|-2} \cdots x^1_{|C_1|+|C_2|-2} \cdots x_{|C_1|+\ldots+|C_{l-1}|-l+2}^{|C_l|-1}\cdots x_{|C_1|+\ldots+|C_{l}|-l}^1.$$
This monomial lies in the ideal $(S_n \cdot x_1^m, S_n \cdot (x_1 x_2 \cdots x_{R_k(\mu)})^k | 1 \leq k \leq m-1)$ if and only if there exists an integer $1 \leq k \leq m$ such that for at least $R_k(\mu)$ distinct indices $j$, $x_j^k$ divides the monomial. For each column $C_i$ we have $\max\{|C_i|-k,0\}$ distinct variables occurring with exponent at least $k$ in the monomial. Now, 
\begin{align*}
    \sum_{i=1}^l\max\{|C_i|-k,0\} &= \sum_{i=1, |C_i| \geq k+1}^l |C_i|-k \\ 
    &= \sum_{t \geq k+1}\sum_{\substack{i=1, \\ |C_i|=t}}^l t-k \\
     &= \sum_{t \geq k+1} (\lambda_t-\lambda_{t+1})(t-k) \\
      &= \sum_{r \geq 1} (\lambda_{k+r}-\lambda_{k+r+1})r \\
      &= \lambda_{k+1}-\lambda_{k+2}+2\lambda_{k+2}-2\lambda_{k+3}+\ldots \\
      &= \lambda_{k+1} + \lambda_{k+2} + \lambda_{k+3} + \ldots \\
      &= R_k(\lambda)-1.
\end{align*}
Since $\# \{ j : (\lambda')_j \geq k \} = \lambda_k$ and because $\mu \not \unlhd \lambda$, there is an integer $s$ with $\mu_1+\ldots+\mu_s > \lambda_1 + \ldots +\lambda_s$. However, this is equivalent to $n - (\mu_1+\ldots+\mu_s) < n - (\lambda_1+\ldots+\lambda_s)$ which in turn is equivalent to $\mu_m + \ldots + \mu_{s+1} + 1 < \lambda_{\len (\lambda)}+\ldots+\lambda_{s+1}+1$. The left hand side equals $R_s(\mu)$ and the right hand side equals $R_s(\lambda)$. Thus, we obtain $R_s(\mu) \leq R_s(\lambda)-1$, as desired.
\end{proof}

In general, both ideal inclusions are strict. The first inclusion is strict already for the partition $(2,1)$, and the second inclusion is strict, for example, for $(3,3,1,1)$.

\section{Classification for $r \leq 2n$}\label{section:classification}

\subsection{Homogeneous symmetric ideals}
Our classification is based on the theory of higher Specht polynomials \cite{Terasoma1993Higher}. These polynomials can be used to produce an explicit isotypic decomposition of the degree~$d$ part $P_d$ of the polynomial ring $P$. This decomposition becomes uniform if the number of variables $n$ is at least $2d$ \cite{Riener2013Symmetries}.

\begin{lemma}\label{lem:classification}
    Let $n \geq 3$. Below we give explicit decompositions of $P_d$ into irreducible $S_n$-representations for $d \leq 3$ using slightly modified higher Specht polynomials. In every direct summand, all indices are meant to run through all of $[n]$ but such that different letters are assigned different values.
    \begin{equation*}
        P_1 = \langle p_1 \rangle  \oplus  \langle x_i - x_j \rangle = S^{(n)} \oplus S^{(n-1,1)}.
    \end{equation*}
    For $n \geq 4$ we have
    \begin{align*}
        P_2 &= \langle p_1^2 \rangle \oplus \langle p_2 \rangle \\
        &\oplus \langle p_1(x_i - x_j) \rangle \oplus \langle x_i^2 - x_j^2 \rangle \\
        &\oplus \langle (x_i - x_j)(x_k - x_l) \rangle \\
        &= 2S^{(n)} \oplus 2S^{(n-1,1)} \oplus S^{(n-2,2)},
    \end{align*}
    while for $n=3$ instead
    \begin{align*}
        P_2 &= \langle p_1^2 \rangle \oplus \langle p_2 \rangle \\
        &\oplus \langle p_1(x_i - x_j) \rangle \oplus \langle x_i^2 - x_j^2 \rangle \\
        &= 2S^{(3)} \oplus 2S^{(2,1)}.
    \end{align*}
    For $n \geq 6$ we have
    \begin{align*}
        P_3 &= \langle p_1^3 \rangle \oplus \langle p_1 p_2 \rangle \oplus \langle p_3 \rangle \\
        &\oplus \langle p_1^2(x_i - x_j) \rangle \oplus \langle p_2(x_i - x_j) \rangle \oplus \langle p_1(x_i^2 - x_j^2) \rangle \oplus \langle x_i^3 - x_j^3 \rangle \\
        &\oplus \langle p_1(x_i-x_j)(x_k-x_l) \rangle \oplus \langle (x_i+x_j+x_k+x_l)(x_i-x_j)(x_k-x_l) \rangle \\
        &\oplus \langle (x_i-x_j)(x_i-x_k)(x_j-x_k) \rangle \\
        &\oplus \langle (x_i-x_j)(x_k-x_l)(x_s-x_t) \rangle \\
        &= 3S^{(n)} \oplus 4S^{(n-1,1)} \oplus 2S^{(n-2,2)} \oplus S^{(n-2,1,1)} \oplus S^{(n-3,3)},
    \end{align*}
    while for $n=5$,
    \begin{align*}
        P_3 &= \langle p_1^3 \rangle \oplus \langle p_1 p_2 \rangle \oplus \langle p_3 \rangle \\
        &\oplus \langle p_1^2(x_i - x_j) \rangle \oplus \langle p_2(x_i - x_j) \rangle \oplus \langle p_1(x_i^2 - x_j^2) \rangle \oplus \langle x_i^3 - x_j^3 \rangle \\
        &\oplus \langle p_1(x_i-x_j)(x_k-x_l) \rangle \oplus \langle (x_i+x_j+x_k+x_l)(x_i-x_j)(x_k-x_l) \rangle \\
        &\oplus \langle (x_i-x_j)(x_i-x_k)(x_j-x_k) \rangle \\
        &= 3S^{(5)} \oplus 4S^{(4,1)} \oplus 2S^{(3,2)} \oplus S^{(3,1,1)},
    \end{align*}
    for $n=4$,
    \begin{align*}
        P_3 &= \langle p_1^3 \rangle \oplus \langle p_1 p_2 \rangle \oplus \langle p_3 \rangle \\
        &\oplus \langle p_1^2(x_i - x_j) \rangle \oplus \langle p_2(x_i - x_j) \rangle \oplus \langle p_1(x_i^2 - x_j^2) \rangle \oplus \langle x_i^3 - x_j^3 \rangle \\
        &\oplus \langle p_1(x_i-x_j)(x_k-x_l) \rangle \\
        &\oplus \langle (x_i-x_j)(x_i-x_k)(x_j-x_k) \rangle \\
        &= 3S^{(4)} \oplus 4S^{(3,1)} \oplus S^{(2,2)} \oplus S^{(2,1,1)},
    \end{align*}
    and for $n=3$,
    \begin{align*}
        P_3 &= \langle p_1^3 \rangle \oplus \langle p_1 p_2 \rangle \oplus \langle p_3 \rangle \\
        &\oplus \langle p_1^2(x_i - x_j) \rangle \oplus \langle p_2(x_i - x_j) \rangle \oplus \langle p_1(x_i^2 - x_j^2) \rangle \\
        &\oplus \langle (x_1-x_2)(x_1-x_3)(x_2-x_3) \rangle \\
        &= 3S^{(3)} \oplus 3S^{(2,1)} \oplus S^{(1,1,1)}.
    \end{align*}
\end{lemma}

\begin{table}
\centerline{
    \begin{tabular}{|c|c|c|c|c|}
        \hline
        Row & $P/I$ & Homogeneous symmetric ideals $I$ & Dim. & Geometry\\
        \Xhline{3\arrayrulewidth}
        1 & $rS^{(n)}$ &  $(x_i-x_j) + \mm^r$ & $r$ & smooth\\
        \hline
        2 & $(r-n+1)S^{(n)} \oplus S^{(n-1,1)}$ & \makecell{
        $(x_i-x_j)_{\geq 2} + \mm^{r-n+1}$ \\
        $(p_1(x_i-x_j),x_i^2-x_j^2,(x_i-x_j)(x_k-x_l)) + \mm^{r-n}$ 
        } & $r-n+2$ & \makecell{singular \\ smooth}\\
        \hline
        3 & $S^{(n)} \oplus S^{(n-1,1)}$ & $(p_1) + \mm^2$ & $2$ & smooth\\
        \hline
        4 & $2S^{(n)} \oplus S^{(n-1,1)}$ & \makecell{
        $\mm^2$\\
        $(p_1,x_i^2-x_j^2,(x_i-x_j)(x_k-x_l)) + \mm^3$
        } & $3$ & \makecell{singular \\ smooth}\\
        \hline
        5 & $3S^{(n)} \oplus S^{(n-1,1)}$ & $(ap_1^2+bp_2,p_1(x_i-x_j),x_i^2-x_j^2,(x_i-x_j)(x_k-x_l)) + \mm^3$ & $4$ & $[-1:n]$ \\
        \hline
        6 & $S^{(n)} \oplus 2S^{(n-1,1)}$ & $(p_1,p_2,(x_i-x_j)(x_k-x_l)) + \mm^3$ & $1$ & smooth\\
        \hline
        7 & $2S^{(n)} \oplus 2S^{(n-1,1)}$ &  \makecell{
        $(p_1,(x_i-x_j)(x_k-x_l)) + \mm^3$ \\
        $(p_1,p_2,(x_i-x_j)(x_k-x_l)) + \mm^4$ \\
        $(p_1^2,p_2,ap_1(x_i-x_j)+b(x_i^2-x_j^2),(x_i-x_j)(x_k-x_l)) + \mm^3$
        } & $4$ & \makecell{ smooth \\ smooth \\ $[1:0]$} \\
        \hline
       8 &  $S^{(5)} \oplus S^{(4,1)} \oplus S^{(3,2)}$ & $(p_1,p_2,x_i^2-x_j^2) + \mm^3$ & $2$ & smooth\\
        \hline
       9 & $S^{(4)} \oplus S^{(3,1)} \oplus S^{(2,2)}$ & $(p_1,p_2,x_i^2-x_j^2) + \mm^3$ & $2$ & smooth\\
        \hline
       10 & $2S^{(4)} \oplus S^{(3,1)} \oplus S^{(2,2)}$ & \makecell{
        $(p_1,x_i^2-x_j^2) + \mm^3$ \\
        $(p_1^2,p_2,p_1(x_i-x_j),x_i^2-x_j^2) + \mm^3$
        } & $3$ & \makecell{ smooth \\ smooth}\\
        \hline
      11 & $3S^{(4)} \oplus S^{(3,1)} \oplus S^{(2,2)}$ & $(ap_1^2+bp_2,p_1(x_i-x_j),x_i^2-x_j^2) + \mm^3$ & $4$ & smooth\\
      \hline
      12 & $S^{(3)} \oplus 2S^{(2,1)} \oplus S^{(1,1,1)}$ & $(p_1,p_2,p_3)$ & $3$ & smooth\\
      \hline
      13 & $2S^{(3)} \oplus 2S^{(2,1)}$ & $(p_1,p_2,(x_1-x_2)(x_1-x_3)(x_2-x_3))$ & $4$ & smooth\\
        \hline
    \end{tabular}
}
\caption{The homogeneous symmetric ideals for $r \leq 2n$\label{tab:classification}}
\end{table}

Table~\ref{tab:classification} provides the first main result of this section. It provides all homogeneous symmetric ideals $I \subseteq P = \CC[x_1,\ldots,x_n]$ for $n \geq 3$ and such that $r = \dim(P/I) \leq 2n$. Note that all unspecified indices in each appearing generator are meant to be \emph{distinct} and run through all of $[n]$. The generators of the form $(x_i-x_j)(x_k-x_l)$ must therefore be interpreted as only occurring for $n \geq 4$ but not for $n=3$. To avoid redundancy in the table, for row~2 one should take $r \geq n+3$. For example, for $r = n+2$ the first ideal of row~2 is the same as that of row~5 for $[a:b] = [-1:n]$. Finally, the second ideal of row~7 has colength $7>2n$ for $n=3$, so we only consider it for $n \geq 4$.

\begin{theorem}\label{thm:classification}
    Let $n \geq 3$. Table~\ref{tab:classification} lists all homogeneous symmetric ideals $I$ with $\dim_\CC P/I = r \leq 2n$ together with the isomorphism class of the $S_n$-module $P/I$. For all occurring representations $\rho$, the invariant Hilbert scheme $\HH_\rho$ is irreducible. Moreover, each point of $\HH_\rho$ corresponding to a homogeneous symmetric ideal is smooth or singular according to the last column of Table~\ref{tab:classification}, where a point $[a:b]$ of $\PP^1$ means that the corresponding ideal is the only singular point among the homogeneous ideals on this invariant Hilbert scheme. The second to last column gives the dimension of $\HH_\rho$.
\end{theorem}

Let us point out that the representation of Row~6 in Table~\ref{tab:classification} is \emph{not} a direct sum of permutation modules. In particular, $\HH_\rho$ can be non-empty even for such $\rho$. In fact, for $\rho = S^{(n)} \oplus 2 S^{(n-1,1)}$, we have $\HH_\rho \cong \CC^1$ as the proof of Theorem~\ref{thm:classification} shows, so the only embedded symmetric deformations of the homogeneous ideal in Row~6 are those obtained from it via the action of $\mathbb{G}_a$.

Recently, Griffin studied the associated graded ideals of vanishing ideals of unions of two distinct $S_n$-orbits as analogs to Tanisaki ideals \cite{griffin2022orbit}. Interestingly, the unique singular point in Row~7(c) for $[a:b] = [1:0]$ is a homogeneous ideal which lies in the smoothable component of $\HH_\rho$ (since $\HH_\rho$ is irreducible) with $\rho = 2S^{(n)} \oplus 2 S^{(n-1,1)}$ but the proof of Theorem~\ref{thm:classification} shows that it cannot be obtained as the associated graded ideal of \emph{any} radical ideal.

\begin{remark}
    The fact that $\mm^2$ and $(x_i-x_j)_{\geq 2} + \mm^d$ for all $d \geq 3$ are singular points of the invariant Hilbert scheme for any $n \geq 3$ in particular recovers the well-known fact that $\mathrm{Hilb}_r(\CC^3)$ is singular for all $r \geq 4$. 
\end{remark}

\begin{proof}[Proof of Classification in Table~\ref{tab:classification}]
    This is a case distinction using Lemma~\ref{lem:classification}. If $I_1 = P_1$, then $I = \mm$, and if $I_1 = \langle x_i-x_j \rangle$, then $I = (x_i-x_j) + \mm^r$ for some $r$, so $P/I = r S^{(n)}$, corresponding to row~1 in Table~\ref{tab:classification}. Hence, we can assume that either $I_1 = \langle p_1 \rangle$ or $I_1 = 0$.\\
    
    We will begin with $I_1 = \langle p_1 \rangle$.
    In this case, $\dim(P/I)_{\leq 1} = n$, so $\dim(P/I)_{\geq 2} = r-n \leq n$. Now, $\dim(S^{(n-2,2)}) = \frac{n(n-3)}{2}$. For $n \geq 6$, this is always $>n$. Hence, the dimension constraint $r \leq 2n$ implies $(x_i-x_j)(x_k-x_l) \in I$ for all $n \geq 6$. Let us, therefore, first assume $(x_i-x_j)(x_k-x_l) \in I$ for any $n \geq 3$ (for $n=3$ this is a vacuous condition).
    
    If both of $p_2$ and $x_i^2-x_j^2$ do not lie in $I$, then necessarily $I = (p_1,(x_i-x_j)(x_k-x_l)) + \mm^3$ since already $\dim(P/I)_{\leq 2} = 2n$. This ideal is the first ideal in row~7 of Table~\ref{tab:classification}. If $p_2 \in I$ but $x_i^2 - x_j^2 \not\in I$, we get either the ideal of row~6 or, for $n \geq 4$, the second ideal of row~7. This follows using Lemma~\ref{lem:relations} and observing that $(x_i-x_j)(x_i-x_k)(x_j-x_k) \in ((x_i-x_j)(x_k-x_l)) \subseteq I$ for all $n \geq 4$, as the following equation shows:
    \begin{align*}
        (x_1-x_2)(x_1-x_3)(x_2-x_3) &= (x_3 - x_1) \cdot (x_1-x_2)(x_3-x_4) \\ &+ (x_1-x_2) \cdot (x_1-x_3)(x_2-x_4).
    \end{align*}
    For $n=3$ there are two more possibilities here, namely rows~12 and~13.
    If both $p_2 \in I$ and $x_i^2 - x_j^2 \in I$, then $I = (p_1) + \mm^2$, row~3. Otherwise, $x_i^2 - x_j^2 \in I$ but $p_2 \not\in I$. Applying the Reynolds operator $P \longrightarrow P^{S_n}, \ f \mapsto \frac{1}{n!}\sum_{\sigma \in S_n}\sigma(f)$ to $(x_1^2 - x_2^2)x_1$ gives a scalar multiple of $n p_3 - p_2p_1$, hence $p_3 \in I$. Therefore, $(P/I)_3$ does not contain any trivial representation. Moreover, $\dim(P/I)_{\leq 2} = n+1$. However, $(x_1-x_2)(x_1-x_3)(x_2-x_3) \in (x_i^2-x_j^2)$ for all $n \geq 3$. Indeed, the formula
    \begin{equation*}
        (x_1-x_2)(x_1-x_3)(x_2-x_3) = x_1(x_3^2-x_2^2) + x_2(x_1^2-x_3^2) + x_3(x_2^2-x_1^2)
    \end{equation*}
    can easily be verified. Moreover, by Lemma~\ref{lem:relations}, for all $n \geq 3$ we have $p_2(x_i-x_j), x_i^3-x_j^3 \in (p_1(x_i-x_j),x_i^2-x_j^2,(x_i-x_j)(x_k-x_l)) \subseteq I$, implying $I_3 = P_3$, so we obtain the second ideal of row~4.
    
    Otherwise, still assuming $I_1 = \langle p_1 \rangle$, we have $n \in \{4,5\}$ and $(x_i-x_j)(x_k-x_l) \not\in I$.
    If $n=5$, then necessarily $I = (p_1,p_2,x_i^2-x_j^2) + \mm^3$ for dimension reasons, which is row~8.
    If $n=4$, then we have $\dim(P/I)_{\leq 2} \geq 6$, so necessarily $x_i^2-x_j^2 \in I$ for dimension reasons. Applying again the Reynolds operator to $(x_1^2-x_2^2)x_1$, we get $p_3 \in I$. Hence, if $p_2 \in I$, we get row~9. If instead $p_2 \not\in I$, then we get the first ideal of row~10. This finishes the case $I_1 = \langle p_1 \rangle$.\\
    
    Secondly, we assume $I_1 = 0$, so $\dim(P/I)_{\leq 1} = n+1$. For $n \geq 5$ and $n=3$, all irreducible $S_n$-representations except for the trivial and the alternating one have dimension $\geq n-1$. Moreover, the least degree $k$ such that $P_k$ contains an alternating representation is $k = \binom{n}{2}$ which is $>n$ for all $n \geq 4$. Hence, for all $n \geq 5$, this forces that $(P/I)_{\geq 2}$ is either $0$, a direct sum of only trivial representations or $S^{(n-1,1)}$, and in the last case (even for $n=3,4$) this forces $I_3 = P_3$ for dimension reasons, so we obtain the third ideal of row~7 for some $(a,b) \neq (0,0)$. If $(P/I)_2 = 0$, then $I = \mm^2$, so we get the first ideal of row~4 (even for $n=3,4$). Otherwise, assume we have any $n \geq 3$ and that $(P/I)_{\geq 2}$ only consists of trivial representations, necessarily of exactly $(r-n-1)$ many. We claim that we can only get the two ideals of row~2 or the ideal of row~5 for some $(a,b) \neq (0,0)$ (the case $(a,b) = \lambda (-1,n)$ gives the same as the first ideal of row~2). The reasoning is as follows. Let $J$ be the ideal generated by all non-trivial representations in all degrees $2 \leq k \leq n-1$ and by $\mm^n$. Then, clearly, $J \subseteq I$. But we also know $J \subseteq (x_i-x_j)_{\geq 2} + \mm^n$. Next, we claim that $(P/J)_k$ has dimension at most $1$ for all $k \geq 3$. For this, observe that $p_1(x_i-x_j) \in I$ and $x_i^k-x_j^k \in I$ for all $k \geq 2$ for dimension reasons. Applying again the Reynolds operator to $p_d(x_1-x_2)x_1$, $d \in \{1,2\}$, and to $(x_1^k-x_2^k)x_1$ gives scalar multiples of $n p_d p_2 - p_d p_1^2$ and $np_{k+1} - p_kp_1$, respectively. The claim $\dim(P/J)_k \leq 1$ for all $k \geq 3$ easily follows from this. For dimension reasons, then, $J_k = (x_i-x_j)_k$ or $I_k = P_k$ for all $k \geq 3$. Hence, the same is true for $I_k$ for all $k \geq 3$. This shows that we indeed get one of the three ideals in rows~2 and~5.
    
    The only remaining cases are $n \in \{3,4\}$ and such that $(P/I)_{\geq 2}$ contains some non-trivial representation which is also different from $S^{(n-1,1)}$. For $n=4$, then, we must have $S^{(2,2)}$ in $(P/I)_2$, i.e., $(x_i-x_j)(x_k-x_l) \not\in I$. Hence, $\dim(P/I)_{\leq 2} \geq 7$. Moreover, both $p_1(x_i-x_j)$ and $x_i^2-x_j^2$ are necessarily in $I$ for dimension reasons, and some linear combination of $p_1^2$ and $p_2$ must also be contained in $I$, again for dimension reasons. Hence, we get either the second ideal of row~10 or the ideal of row~11 for some $(a,b) \neq (0,0)$. Finally, for $n=3$, the only possibility for a non-trivial representation in $(P/I)_{\geq 2}$ not equal to $S^{(2,1)}$ is the alternating representation $S^{(1,1,1)}$ in degree $3$, generated by $(x_1-x_2)(x_1-x_3)(x_2-x_3)$. Then $I$ contains $(p_1(x_i-x_j),x_i^2-x_j^2)$ for dimension reasons. But we have already seen that then automatically $(x_1-x_2)(x_1-x_3)(x_2-x_3) \in I$ as well, so this is impossible.
\end{proof}

\begin{lemma}\label{lem:relations}
    For all $n \geq 3$ we have
    \begin{gather*}
         x_1^3-x_2^3 \in (x_n p_1(x_1-x_2), p_1(x_1^2-x_2^2), p_2(x_1-x_2), (x_i-x_j)(x_k-x_l)), \\
        p_2(x_1-x_2), x_1^3-x_2^3 \in (x_n p_1(x_1-x_2),x_i^2-x_j^2,(x_i-x_j)(x_k-x_l)).
    \end{gather*}
\end{lemma}

\begin{proof}
    All claims follow from the containments $f,g \in ((x_i-x_j)(x_k-x_l))$, where
    \begin{align*}
    f &\coloneqq n(n-1)(x_1^3-x_2^3) - (n^2-3n+3)p_2(x_1-x_2) \\ &- (2n-3)p_1(x_1^2-x_2^2) + n(n-2)x_n p_1(x_1-x_2), \\
    g &\coloneqq (n-2)p_2(x_1-x_2) - n x_n p_1(x_1-x_2) + n x_n (x_1^2-x_2^2) \\ &- n(n-2)(x_1-x_2)(x_1^2-x_n^2) + (n-2)(x_1-x_2) \sum_{i=2}^n (x_1^2-x_i^2).
    \end{align*}
    The ideal $((x_i-x_j)(x_k-x_l))$ is the vanishing ideal of the set $Z \coloneqq \{S_n.(b,a,a,\ldots,a) : a,b \in \CC\}$, in particular it is radical. The radicality follows from the fact that $((x_i-x_j)(x_k-x_l))$ is the Specht ideal for the partition $(n-2,2)$, and all Specht ideals are radical by~\cite{Murai2022Specht}.\footnote{This was first proved by Haiman and Woo in an unpublished manuscript titled ``Garnir modules, Springer fibers, and Ellingsrud--Str{\o}mme cells on the Hilbert Scheme of points.''} Hence, it is enough to show that $f$ and $g$ vanish on  $Z$. Since both are divisible by $x_1-x_2$, it is enough to prove their vanishing on $(b,a,a,\ldots,a)$ and $(a,b,a,\ldots,a)$. These are simple computations.
\end{proof}

\subsection{Singular points in Table~\ref{tab:classification}}
We now indicate how to obtain the last column of Table~\ref{tab:classification}, i.e., how to decide which ideals correspond to smooth points of $\HH_\rho$. For all ideals in rows~8--13 we have checked smoothness via \texttt{Macaulay2}, implementing the algorithm of Lehn and Terpereau \cite[Section~5]{LehnTerpereau} for computing the tangent space. For smoothness one only has to check that the tangent space dimension agrees with the dimension of the smoothable component and argue that it contains the respective ideal which is the case for all ideals in Table~\ref{tab:classification}. The theoretical background of the alogrithm of Lehn and Terpereau also provides the foundation for our proofs for rows~1--7, so we recall some basic results of \emph{loc. cit.} here.

Let $I \subseteq P = \CC[x_1, \ldots, x_n]$ be a homogeneous symmetric ideal. Let $N_1 \subseteq P$ be a minimal graded $S_n$-submodule generating $I$ as a $P$-module, and let $n_1 \coloneqq \dim_\CC(N_1)$. Note that each graded piece of $N_1$ is an $S_n$-representation and that $n_1$ is usually larger than $\beta_0(I)$, the minimal number of generators of $I$.

Let $K \coloneqq \ker(P \otimes_\CC N_1 \rightarrow I)$. Observe that $K$ is an $S_n$-module via the inclusion into $P \otimes_\CC N_1$, where $S_n$ acts diagonally on $P \otimes_\CC N_1$. Let now $N_2 \subseteq P \otimes_\CC N_1$ be \emph{any} graded subvector space generating $K$ as a $P$-module, giving a surjection $P \otimes N_2 \rightarrow K$. Let $n_2 \coloneqq \dim_\CC(N_2)$. 

\begin{remark}
    Lehn and Terpereau require $N_2 \subseteq P \otimes_\CC N_1$ to be a minimal graded $S_n$-submodule generating $K$ as a $P$-module. This is indeed important for the rest of their algorithm but both minimality and the $S_n$-module structure of $N_2$ are actually not needed for the tangent space computation alone. In fact, even the minimality of $N_1$ is not strictly necessary, only the grading and the $S_n$-module structure are.
\end{remark}

From the short exact sequence $0 \rightarrow K \rightarrow P \otimes_\CC N_1 \overset{\alpha}{\rightarrow} I \rightarrow 0$ and the definition of $N_2$ we get a commutative diagram
\begin{equation*}
    \begin{tikzcd}
    0 \arrow[r]
        &\Hom_P(I,P/I) \arrow[r,"\alpha^\ast"]
            &\Hom_P(P \otimes N_1, P/I) \arrow[r] \arrow[dr,"\beta^\ast"']
                &\Hom_P(K,P/I) \arrow[d,hook] \\
        &
            &
                &\Hom_P(P \otimes N_2, P/I),
\end{tikzcd}
\end{equation*}

where the upper row is exact and $\beta: P \otimes N_2 \rightarrow P \otimes N_1$ is the canonical map. Now, the presence of the Reynolds operator $P \longrightarrow P^{S_n}, \ f \mapsto \frac{1}{n!}\sum_{\sigma \in S_n}\sigma(f)$ yields that taking $S_n$-invariants is an exact functor. Therefore, $\alpha^\ast$ induces an isomorphism 
\begin{equation*}
    \Hom^{S_n}_P(I,P/I) \cong \ker(\beta^\ast) \cap \Hom^{S_n}_P(P \otimes N_1, P/I).
\end{equation*}
If $P/I \cong \rho$ as $S_n$-modules, the left hand side is naturally isomorphic to $T_{[I]} \HH_\rho$, and the right hand side is isomorphic to
\begin{equation*}
    \ker(\beta^\ast) \cap \Hom_{S_n}(N_1,\rho),
\end{equation*}
giving an effective method to compute the dimension of $T_{[I]} \HH_\rho$, since $\beta^\ast$ can be represented by an $n_2 \times n_1$ matrix with entries in $P/I$, after choosing bases of $N_1$ and $N_2$.

\subsubsection{Row~1} $I = (x_i - x_j) + \mm^r$. Here, $N_1 = \langle x_i - x_j \rangle \oplus \langle p_1^r \rangle$ and $P/I = r S^{(n)} = \langle 1 \rangle \oplus \langle p_1 \rangle \oplus \cdots \oplus \langle p_1^{r-1} \rangle$, hence any $f \in \Hom_{S_n}(N_1,P/I)$ satisfies $f(x_i - x_j) = 0$ and $f(p_1^d) = \sum_{i=0}^{r-1} \alpha_i p_1^i$ for some $\alpha_i \in \CC$. In particular, $\dim_\CC(T_{[I]} \HH) \leq r$. On the other hand, $I$ lies in the smoothable component of $\HH$, whose dimension is $r$. The last claim follows by noticing that the vanishing ideal of $r$ distinct diagonal points $(a_1, \ldots, a_1), \ldots, (a_r, \ldots, a_r)$ contains $(x_i - x_j)$ and the product of the $r$ polynomials $p_1 - n a_i$ whose top degree part is precisely $p_1^r$. 

\subsubsection{Row~2(a)+4(a)+5} $I = (x_i - x_j)_{\geq 2} + \mm^{d}$ for $d \geq 2$. Note that for $d=2$ this is the first ideal of row~4 while for $d=3$ we get the ideal of row~5 corresponding to $[a:b] = [-1:n]$.
Next, we observe that $P/(x_i-x_j)_{\geq 2}$ has constant Hilbert function $1$ in all degrees $\geq 2$ and that a homogeneous polynomial of degree $\geq 2$ belongs to $(x_i-x_j)_{\geq 2}$ if and only if it vanishes at $(1,1,\ldots,1)$.
A minimal $S_n$-stable generating set of $I$ is then given by
\begin{equation*}
    N_1 = \langle x_i^2 - x_j^2 \rangle \oplus \langle p_1(x_i - x_j)  \rangle \oplus \langle (x_i - x_j)(x_k - x_l) \rangle \oplus \langle p_1^2 - n p_2 \rangle \oplus \langle p_1^d \rangle.
\end{equation*}
Moreover, $P/I = d S^{(n)} \oplus S^{(n-1,1)} = \langle 1 \rangle \oplus \langle p_1 \rangle \oplus \cdots \oplus \langle p_1^{d-1} \rangle \oplus \langle x_i - x_j \rangle$.
Hence, any $f \in \Hom_{S_n}(N_1,P/I)$ is of the form
\begin{align*}
    f((x_i - x_j)(x_k - x_l)) &= 0, \\
    f(x_i^2 - x_j^2) &= \alpha (x_i - x_j), \\
    f(p_1 (x_i - x_j)) &= \beta (x_i - x_j), \\
    f(p_1^2 - n p_2) &= \sum_{i=0}^{d-1} \gamma_i p_1^i, \\
    f(p_1^d) &= \sum_{i=0}^{d-1} \delta_i p_1^i.
\end{align*}

If this $f$ lies in $\Hom^{S_n}_P(I,P/I)$, then necessarily $\gamma_0 = \gamma_1 = \cdots = \gamma_{d-2} = 0$. Indeed, applying the Reynolds operator to $p_1(x_1-x_2)x_1 \in I$ gives a non-zero rational multiple of $p_1(p_1^2 - n p_2)$, hence $f$ maps the latter to both a non-zero rational multiple of $\sum_{i=0}^{d-1} \gamma_i p_1^{i+1}$ and also to a multiple of the Reynolds operator applied to $\beta (x_1 - x_2) x_1$, which in turn is a multiple of $p_1^2 - n p_2$ and hence $0$ in $P/I$.

We claim that also $\delta_0 = 0$. For this, observe that $p_1(x_1-x_2) \cdot p_1^{d-1}$ maps to both $\beta(x_1-x_2) \cdot p_1^{d-1} = 0 \pmod I$ and to $(x_1 - x_2) \sum_{i=0}^{d-1} \delta_i p_1^i = \delta_0(x_1-x_2) \pmod I$, hence $\delta_0 = 0$.

So far we have proved $\dim_\CC(\Hom^{S_n}_P(I,P/I)) \leq d+2$. We claim that equality holds. For this, observe that there is no linear dependence among the degree $2$ generators of $N_1$ and that every homogeneous multiple of $(x_i-x_j)$ and $p_1^{d-1}$ with strictly larger degree already lies in $I$. Therefore, there is no linear relation among the remaining $\alpha,\beta,\gamma_{d-1},\delta_1,\ldots,\delta_{d-1}$ which actually involves at least one of $\alpha, \beta, \gamma_{d-1}$. Therefore, the only possible linear relations are among $\delta_1, \ldots, \delta_{d-1}$. By what we have explained above, such a relation arises from a syzygy of the generators of $N_1$ such that the coefficient $g$ of $p_1^d$ is a homogeneous polynomial which is non-zero in $P/I$. If $\deg(g) \geq 2$, then after rescaling we have $g = p_1^k \pmod I$ for some $2 \leq k \leq d-1$. But such a syzygy does not exist since no power $p_1^k$ (including $k = 0$ and $k = 1$) is contained in $(x_i - x_j)_{\geq 2}$. Therefore, the only case left is $g \in \langle x_i - x_j \rangle$. But this gives precisely the relation we have already used above to show $\delta_0 = 0$. This proves $\dim_\CC(\Hom^{S_n}_P(I,P/I)) = d+2$.

Finally, in order to show that $I$ is actually a singular point of $\HH_\rho$, where $\rho = d S^{(n)} \oplus S^{(n-1,1)}$, we must show that $I$ lies in the smoothable component which is easily seen to have dimension $d+1$. For this, let $a, b, c_2, c_3, \ldots, c_d \in \CC$ such that $a \neq b$ and $c_2, c_3, \ldots, c_d$ are all distinct. Define $J_{(a,b,c_2,c_3,\ldots,c_d)}$ to be the vanishing ideal of the points given by the $S_n$-orbit of $(b,a,a,\ldots,a)$ and the $d-1$ diagonal points $(c_2,\ldots, c_2), \ldots, (c_d, \ldots, c_d)$. The product of the $d$ polynomials $p_1 - ((n-1)a + b)$ and $p_1 - n c_i$ clearly lies in $J$, and its top degree part is precisely $p_1^d$. Also, for all degree $2$ generators of $N_1$ except for $p_1^2 - np_2$ there is a polynomial in $J$ whose top degree part is precisely this generator. This is evidenced by the polynomials $(x_i - x_j)(x_i + x_j - (a+b))$, $(x_i - x_j)(p_1 - ((n-1)a + b))$ and $(x_i - x_j)(x_k - x_l)$. Fixing $c_2, \ldots, c_d$ and letting $(a,b)$ tend to $(0,0)$ linearly, i.e., along the family $(ta,tb)$, we obtain the ideal $J_{(c_2,\ldots,c_d)} = ((p_1) + \mm^2) \cap \mm_{(c_2,\ldots,c_2)} \cap \cdots \cap \mm_{(c_d,\ldots,c_d)}$. For all distinct $c_2, \ldots, c_d$, this ideal therefore lies in the smoothable component. Moreover, it still contains the above polynomials with $(a,b)$ replaced by $(0,0)$. We claim that $J_{(c_2,\ldots,c_d)}$ also contains $p_1^2 - n p_2$. But this is clear since this polynomial is obviously contained in $(p_1) + \mm^2$ and vanishes on all diagonal points. Hence, letting $(c_2, \ldots, c_d)$ tend to $(0,\ldots, 0)$ linearly, i.e., along the family $(tc_2,\ldots, tc_d)$ which lies entirely in the smoothable component, the limiting ideal contains $p_1^2 - np_2$ and all the above polynomials with all parameters replaced by $0$, which is a generating set for $I$. Therefore, this limiting ideal must be $I$ itself, and hence $I$ lies in the smoothable component.

\medskip
\emph{Row 2(b).} $I = (p_1(x_i-x_j),x_i^2-x_j^2,(x_i-x_j)(x_k-x_l)) + \mm^{d}$ with $d \geq 3$. We first note that $(x_i-x_j)_{\geq 3} \subseteq I$, hence a minimal $S_n$-stable generating set of $I$ is given by
\begin{equation*}
    N_1 = \langle x_i^2 - x_j^2 \rangle \oplus \langle p_1(x_i - x_j)  \rangle \oplus \langle (x_i - x_j)(x_k - x_l) \rangle \oplus \langle p_1^d \rangle.
\end{equation*}
Observe that this is the same as for \emph{Row 2(a)} omitting the generator $p_1^2 - n p_2$. A simpler version of the previous argument then shows that $I$ lies in the smoothable component; in fact, $I$ is the associated graded ideal of \emph{every} radical ideal in $\HH_\rho$. Moreover, the same tangent space computations as for \emph{Row 2(a)} show $\dim_\CC(\Hom^{S_n}_P(I,P/I)) \leq d+1$. Indeed, the argument for $\delta_0 = 0$ still works, and this time the $\gamma_i$ do not appear. This proves that $I$ is a smooth point of $\HH_\rho$.

\subsubsection{Row 3} $I = (p_1) + \mm^2$ is the Tanisaki ideal for the partition $\lambda = (n-1,1)$ of length $2$, so $I$ is a smooth point by Proposition~\ref{prop:m=2}.

\medskip
\subsubsection{Row 4(b)} $I = (p_1,x_i^2-x_j^2,(x_i-x_j)(x_k-x_l)) + \mm^3$. For all $n \geq 3$, $\mm^3$ is already contained in $(p_1,x_i^2-x_j^2,(x_i-x_j)(x_k-x_l))$ by the proof of the classification above. As $\rho = P/I = \langle 1 \rangle \oplus \langle p_2 \rangle \oplus \langle x_i - x_j \rangle = 2S^{(n)} \oplus S^{(n-1,1)}$, picking $N_1 = \langle p_1 \rangle \oplus \langle x_i^2 - x_j^2 \rangle \oplus \langle (x_i - x_j)(x_k - x_l) \rangle$, any $f \in \Hom^{S_n}_{P}(I,P/I)$ has the form
\begin{align*}
    f(p_1) &= \alpha \cdot 1 + \beta p_2, \\
    f(x_i^2 - x_j^2) &= \gamma (x_i - x_j), \\
    f((x_i-x_j)(x_k-x_l)) &= 0,
\end{align*}
so $\dim_\CC(\Hom^{S_n}_P(I,P/I)) \leq 3$. Since the dimension of the smoothable component of $\HH_\rho$ is $3$ as well, it suffices to show that $I$ lies in the latter. For this, pick a radical ideal as in \emph{Row~4(a)} but this time such that $(n-1)a+b = nc$. The associated graded ideal is then $I$.

\subsubsection{Row~5} $I = (ap_1^2+bp_2,p_1(x_i-x_j),x_i^2-x_j^2,(x_i-x_j)(x_k-x_l)) + \mm^3$. In case $[a:b] = [-1:n]$, $I$ is the ideal of \emph{Row~2(a)} with $d = 3$, so we know that this point is singular. We claim that $I$ is a smooth point for all different $[a:b]$. The point here is that the $2$-dimensional torus $T$ fixes this $1$-dimensional set of ideals, and there are precisely two $T$-fixed points, namely $[a:b] = [-1:n]$ and $[a:b] = [1:0]$. It is therefore enough to prove that the ideal obtained for $[a:b] = [1:0]$ corresponds to a smooth point of $\HH_\rho$.
First, we note that $(p_1^2,p_1(x_i - x_j)) = (p_1)_{\geq 2}$, hence $\mm^3$ is already contained in $I = (p_1^2,p_1(x_i-x_j),x_i^2-x_j^2,(x_i-x_j)(x_k-x_l))$, as was the case in \emph{Row~4(b)}.
Hence, we can choose the generators of $I$ as an $S_n$-stable basis for $N_1$. Moreover, $\rho = P/I = \langle 1 \rangle \oplus \langle p_1 \rangle \oplus \langle p_2 \rangle \oplus \langle x_i-x_j \rangle = 3S^{(n)} \oplus S^{(n-1,1)}$. Hence, any $f \in \Hom^{S_n}_{P}(I,P/I)$ has the form
\begin{align*}
    f(p_1^2) &= \alpha \cdot 1 + \beta p_1 + \gamma p_2, \\
    f(p_1(x_i - x_j)) &= \delta (x_i - x_j), \\
    f(x_i^2 - x_j^2) &= \epsilon (x_i - x_j), \\
    f((x_i-x_j)(x_k-x_l)) &= 0.
\end{align*}
But necessarily $\alpha = 0$ since $f$ maps $p_1^2(x_i - x_j)$ to both $\alpha (x_i - x_j)$ and $\delta p_1 (x_i - x_j) = 0 \pmod I$. Hence, $\dim_\CC(\Hom^{S_n}_P(I,P/I)) \leq 4$, and the dimension of the smoothable component of $\HH_\rho$ is $4$ as well.
So finally we prove that $I$ lies in the smoothable component. For this we can pick the vanishing ideal $J$ of two distinct diagonal points $(c,\ldots,c), (d, \ldots, d)$ and the $S_n$-orbit of $(b,a,a,\ldots,a)$ for $a \neq b$ where we assume $(n-1)a+b = nc$. Then, clearly, $(p_1 - nc)(p_1 - nd) \in J$ with top degree part $p_1^2$. Hence, the associated graded ideal of $J$ must be $I$.

\subsubsection{Row~6} $I = (p_1,p_2,(x_i-x_j)(x_k-x_l)) + \mm^3$. We assume $n \geq 4$; the case $n=3$ is similar if one replaces all occurrences of $(x_i-x_j)(x_k-x_l)$ by $(x_1-x_2)(x_1-x_3)(x_2-x_3)$. Here, $N_1 = \langle p_1, p_2, p_3, (x_i-x_j)(x_k-x_l)\rangle$ and $\rho = P/I = \langle 1 \rangle \oplus \langle x_i - x_j \rangle \oplus \langle x_i^2 - x_j^2 \rangle = S^{(n)} \oplus 2S^{(n-1,1)}$. Hence, any $f \in \Hom^{S_n}_{P}(I,P/I)$ has the form
\begin{align*}
    f(p_1) &= \alpha \cdot 1, \\
    f(p_2) &= \beta \cdot 1, \\
    f(p_3) &= \gamma \cdot 1, \\
    f((x_i-x_j)(x_k-x_l)) &= 0.
\end{align*}
We use the containment
\begin{equation*}
    p \coloneqq x_1 x_2 (x_1-x_2) p_1 - (x_1^2 - x_2^2) p_2 + (x_1 - x_2) p_3 \in ((x_i-x_j)(x_k-x_l)),
\end{equation*}
which can be proven as before using the radicality of $((x_i-x_j)(x_k-x_l))$. Hence,
\begin{equation*}
    0 = f(p) = - \beta (x_1^2 - x_2^2) + \gamma (x_1 - x_2).
\end{equation*}
This implies $\beta = \gamma = 0$, hence $\dim_\CC(\Hom^{S_n}_P(I,P/I)) \leq 1$. On the other hand, the additive group $(\CC,+)$ acts freely on $\HH_\rho$ by translation of the support along a diagonal vector in $\CC^n$, so every irreducible component of $\HH_\rho$ has dimension $\geq 1$. Therefore, $I$ must be a smooth point even though there is \emph{no} radical ideal in $\HH_\rho$ by Proposition~\ref{lemma:radicalIdeals}. In fact, since $I$ is the only homogeneous ideal in $\HH_\rho$, this proves that $\HH_\rho \cong \CC^1$ is precisely the $(\CC,+)$-orbit of $I$. So in some sense, $I$ has no meaningful embedded symmetric deformations. 

\subsubsection{Row~7(a)} $I = (p_1,(x_i-x_j)(x_k-x_l)) + \mm^3$. The tangent space dimension for $n=3$ can be computed to be $4$ in \texttt{Macaulay2} via the algorithm of Lehn--Terpereau, so we may assume $n \geq 4$. In this case a minimal $S_n$-stable set of generators is given by $N_1 = \langle p_1, p_3, p_2(x_i-x_j), (x_i-x_j)(x_k-x_l) \rangle$. Moreover, $\rho = P/I = \langle 1 \rangle \oplus \langle p_2 \rangle \oplus \langle x_i - x_j \rangle \oplus \langle x_i^2 - x_j^2 \rangle = 2S^{(n)} \oplus 2S^{(n-1,1)}$.
Hence, any $f \in \Hom^{S_n}_{P}(I,P/I)$ has the form
\begin{align*}
    f(p_1) &= \alpha \cdot 1 + \beta p_2, \\
    f(p_3) &= \gamma \cdot 1 + \delta p_2, \\
    f(p_2(x_i-x_j)) &= \epsilon (x_i-x_j), \\
    f((x_i-x_j)(x_k-x_l)) &= 0.
\end{align*}
We claim $\alpha = 0$. For this, we again look at the first polynomial in the proof of Lemma~\ref{lem:relations}:
\begin{align*}
    p &\coloneqq n(n-1)(x_1^3-x_2^3) - (n^2-3n+3)p_2(x_1-x_2) \\ &- (2n-3)p_1(x_1^2-x_2^2) + n(n-2)x_n p_1(x_1-x_2) \in ((x_i-x_j)(x_k-x_l)).
\end{align*}
Therefore, $f(p) = 0$. At the same time, $x_1^3 - x_2^3 \in I$ spans a copy of $S^{(n-1,1)}$, so it must map to $\epsilon' (x_1 - x_2)$ for some $\epsilon'$. Hence,
\begin{align*}
    0 = f(p) &= \left(n(n-1) \epsilon' - (n^2-3n+3) \epsilon \right) (x_1-x_2) \\ &+ n(n-2) \alpha x_n(x_1-x_2) - (2n-3)(x_1^2 - x_2^2) \alpha.
\end{align*}
Hence, the coefficient of $x_1-x_2$ must vanish. Moreover, if $\alpha$ was non-zero, then necessarily
\begin{equation*}
    g \coloneqq n(n-2) x_n(x_1-x_2) - (2n-3)(x_1^2 - x_2^2) \in I.
\end{equation*}
This is equivalent to the existence of a linear form $h$ with $g - p_1 h \in ((x_i-x_j)(x_k-x_l))$. Since the latter ideal is the vanishing ideal of the set $\{S_n.(b,a,a,\ldots,a) : a,b \in \CC\}$, we only need to show that there is no linear form $h$ such that $g - p_1 h$ vanishes on this entire set. The polynomial $g$ vanishes on all the points with $b$ not occurring as the first or second coordinate. A simple computation shows that then necessarily $h = \lambda (x_1-x_2)$ for some $\lambda \in \CC$. Another computation shows that there is no $\lambda$ such that $g - \lambda p_1 (x_1-x_2)$ vanishes on $(b,a,a,\ldots,a)$ for all $a,b \in \CC$, a contradiction. This shows $\alpha = 0$ and hence $\dim_\CC(\Hom^{S_n}_P(I,P/I)) \leq 4$.

Since the smoothable component of $\HH_\rho$ has dimension $4$ as well, we only need to show that $I$ is smoothable. Indeed, let $J$ be the vanishing ideal of the two $S_n$-orbits of $(b,a,a,\ldots,a)$ and $(d,c,c,\ldots,c)$ where $a \neq b$, $c \neq d$ and $p_1(b,a,a,\ldots,a) = p_1(d,c,c,\ldots,c)$. The associated graded ideal of $J$ then clearly contains $p_1$ and $(x_i - x_j)(x_k - x_l)$, hence must be $I$.

\medskip
\emph{Row~7(b).} $I = (p_1,p_2,(x_i-x_j)(x_k-x_l)) + \mm^4$ for $n \geq 4$. Here, $\mm^4$ is already contained in $(p_1,p_2,(x_i-x_j)(x_k-x_l))$, so these generators form a minimal $S_n$-stable basis of $N_1$. Moreover, $\rho = P/I = \langle 1 \rangle \oplus \langle p_3 \rangle \oplus \langle x_i - x_j \rangle \oplus \langle x_i^2 - x_j^2 \rangle = 2S^{(n)} \oplus 2 S^{(n-1,1)}$.
Hence, any $f \in \Hom^{S_n}_{P}(I,P/I)$ has the form
\begin{align*}
    f(p_1) &= \alpha \cdot 1 + \beta p_3, \\
    f(p_2) &= \gamma \cdot 1 + \delta p_3, \\
    f((x_i-x_j)(x_k-x_l)) &= 0,
\end{align*}
yielding $\dim_\CC(\Hom^{S_n}_P(I,P/I)) \leq 4$. Therefore, it suffices again to prove that $I$ is smoothable. One may pick any radical ideal $J$ as we did for \emph{Row~7(a)} but such that additionally $p_2(b,a,a,\ldots,a) = p_2(d,c,c,\ldots,c)$ holds. Then, the associated graded ideal of $J$ contains $p_1$, $p_2$ and $(x_i-x_j)(x_k-x_l)$, hence must be $I$.

\medskip
\emph{Row~7(c).} $I = (p_1^2,p_2,ap_1(x_i-x_j)+b(x_i^2-x_j^2),(x_i-x_j)(x_k-x_l)) + \mm^3$. As for \emph{Row~5} above, the $2$-dimensional torus $T$ fixes this $1$-dimensional set of ideals and there are precisely two $T$-fixed points, namely $[a:b] \in \{[1:0],[-2:n]\}$. It suffices to prove that the first choice yields a singular point while the second choice yields a smooth point. For $n=3$ this can again be checked in \texttt{Macaulay2}, so we may assume $n \geq 4$.
We first assume $[a:b] = [1:0]$. We let $N_1 = \langle p_1^2, p_2, p_3, p_1(x_i-x_j), (x_i-x_j)(x_k-x_l) \rangle$ and $\rho = P/I = \langle 1 \rangle \oplus \langle p_1 \rangle \oplus \langle x_i - x_j \rangle \oplus \langle x_i^2 - x_j^2 \rangle = 2S^{(n)} \oplus 2S^{(n-1,1)}$.
Hence, any $f \in \Hom^{S_n}_{P}(I,P/I)$ has the form
\begin{align*}
    f(p_1^2) &= \alpha_0 \cdot 1 + \alpha_1 p_1, \\
    f(p_2) &= \beta_0 \cdot 1 + \beta_1 p_1, \\
    f(p_3) &= \gamma_0 \cdot 1 + \gamma_1 p_1, \\
    f(p_1(x_i-x_j)) &= \delta_0 (x_i-x_j) + \delta_1 (x_i^2-x_j^2), \\
    f((x_i-x_j)(x_k-x_l)) &= 0.
\end{align*}
We claim $\alpha_0 = \beta_0 = \gamma_0 = 0$. This follows again from the containment
\begin{equation*}
    p_3(x_1-x_2) - p_2(x_1^2-x_2^2) + p_1 x_1 x_2 (x_1-x_2) \in ((x_i-x_j)(x_k-x_l)).
\end{equation*}
Applying $f$, we obtain
\begin{equation*}
    0 = \gamma_0 (x_1 - x_2) - \beta_0 (x_1^2 - x_2^2) + \frac{\alpha_0}{n} x_2(x_1-x_2).
\end{equation*}
Here, we have used $n p_1 x_1 = p_1^2 + \sum_{i=2}^n (x_1 - x_i) p_1$ for the last term. Note moreover that $p_1 x_1 x_2 (x_1-x_2).$ must also map to $\delta_0 (x_1-x_2)x_1x_2 + \delta_1 (x_1^2-x_2^2)x_1x_2 = 0 \pmod I$, because $\mm^3 \subseteq I$. Therefore, $\alpha_0 = 0$ since $x_2(x_1-x_2) \not\in I$ because otherwise $x_1^2 - x_2^2 = x_1(x_1-x_2) + x_2(x_1-x_2) \in I$ which is not the case. Finally, we also get $\beta_0 = \gamma_0 = 0$, as claimed. 

We claim next that no more relations among $\alpha_1,\beta_1,\gamma_1,\delta_0,\delta_1$ exist. Clearly, there is no syzygy in degree $2$. For a syzygy in degree $3$, the coefficient of $p_3$ must be zero, otherwise $p_3$ would be contained in the ideal generated by the remaining generators of $N_1$. Moreover, the only linear relation among $\alpha_1,\beta_1,\gamma_1,\delta_0,\delta_1$ which could follow from a syzygy in degree $3$ would be $\delta_0 = 0$ because $\mm^3 \subseteq I$ and $(p_1)_{\geq 2} \subseteq I$. In fact, $\delta_0 = 0$ is equivalent to the existence of a linear form $g$ such that $p_1(x_1-x_2)g \in (p_1^2, p_2, (x_i-x_j)(x_k-x_l))_3$ and $g(x_1-x_2) \not\in I$. We claim that such a $g$ does not exist.

To see this, we may assume $g \in \langle x_i - x_j \rangle$. Note that if $g$ does not involve the variables $x_1$ and $x_2$, then $(x_1-x_2)g \in ((x_i-x_j)(x_k-x_l)) \subseteq I$, so we may assume that $g$ involves \emph{only} $x_1$ and $x_2$. The containment $p_1(x_1-x_2)g \in (p_1^2, p_2, (x_i-x_j)(x_k-x_l))_3$ is equivalent to the existence of two linear forms $g',g''$ such that
\begin{equation*}
    p_1(x_1-x_2)g + p_1^2g' + p_2 g'' \in ((x_i-x_j)(x_k-x_l)),
\end{equation*}
which in turn is equivalent to the vanishing of this polynomial on the vanishing set of $((x_i-x_j)(x_k-x_l))$. Picking a point $(a,a,b,a,\ldots,a)$ where $b$ is not in position $1$ or $2$, the first summand vanishes and we get
\begin{equation*}
    ((n-1)a+b)^2g'(a,a,b,a,\ldots,a) + ((n-1)a^2+b^2)g''(a,a,b,a,\ldots,a) = 0
\end{equation*}
for all $a,b \in \CC$, and all analogous equations where $b$ is at the $i$-th position with $3 \leq i \leq n$. Hence, in all these $n-2$ equations all coefficients of $a^3,a^2b,ab^2,b^3$, which are linear forms in the coefficients of $g'$ and $g''$, must vanish. A computation shows that the only way in which this is possible is that $g' = \mu (x_1-x_2)$ and $g'' = \nu (x_1-x_2)$ for some $\mu, \nu \in \CC$. Coming back to $p_1(x_1-x_2)g + p_1^2g' + p_2 g''$, applying the transposition $(1,2)$ to this polynomial, adding the result to it and evaluating on $(b,a,a,\ldots,a)$ shows that $g = \lambda (x_1 + x_2)$ for some $\lambda \in \CC$. Hence, the polynomial
\begin{equation*}
    \lambda p_1 (x_1+x_2) + \mu p_1^2 + \nu p_2
\end{equation*}
must vanish on $(b,a,a,\ldots,a)$ for all $a,b \in \CC$. It can be checked that this is only possible if $\lambda = \mu = \nu = 0$. In particular, we obtain $g = 0$, so a $g$ with the desired properties does not exist.

Finally, any syzygy in degree $\geq 4$ does not yield any linear relation among the remaining $\alpha_1,\beta_1,\gamma_1,\delta_0,\delta_1$, again because $\mm^3 \subseteq I$ and $(p_1)_{\geq 2} \subseteq I$. Thus, we have proved $\dim_\CC(\Hom^{S_n}_P(I,P/I)) = 5$.

\smallskip
We now consider the ideal $I$ corresponding to $[-2:n]$. Note that it is enough to prove $\dim_\CC(\Hom^{S_n}_P(I,P/I)) \leq 4$ and that $I$ lies in the smoothable component in order to see that $I$ is a smooth point of $\HH_\rho$. Then necessarily \emph{all} ideals in the one-dimensional family lie in the smoothable component. Moreover, all of them must be smooth points of $\HH_\rho$ except for the other $T$-fixed point corresponding to $[1:0]$ above, which then must be singular.

First, $\mm^3$ is already contained in the ideal generated by
\begin{equation*}
    N_1 = \langle p_1^2, p_2, -2p_1(x_i - x_j) + n(x_i^2 - x_j^2), (x_i-x_j)(x_k-x_l) \rangle.
\end{equation*}
Moreover, $\rho = P/I = \langle 1 \rangle \oplus \langle p_1 \rangle \oplus \langle x_i-x_j \rangle \oplus \langle p_1(x_i-x_j) \rangle = 2S^{(n)} \oplus 2S^{(n-1,1)}$.
Hence, any $f \in \Hom^{S_n}_{P}(I,P/I)$ has the form
\begin{align*}
    f(p_1^2) &= \alpha_0 \cdot 1 + \alpha_1 p_1, \\
    f(p_2) &= \beta_0 \cdot 1 + \beta_1 p_1, \\
    f(-2p_1(x_i-x_j) + n(x_i^2 - x_j^2)) &= \gamma_0 (x_i-x_j) + \gamma_1 p_1(x_i-x_j), \\
    f((x_i-x_j)(x_k-x_l)) &= 0.
\end{align*}
We use the containment
\begin{align*}
    p &\coloneqq (n-2) p_1^2(x_1-x_2) - n(n-2)p_2(x_1-x_2) \\ &- (n-1)(-2p_1(x_1-x_2)+n(x_1^2-x_2^2))(x_1-x_2) \\
    &+ 2(n-1)(-2p_1(x_1-x_3)+n(x_1^2-x_3^2))(x_1-x_2) \in ((x_i-x_j)(x_k-x_l)).
\end{align*}
Hence,
\begin{align*}
    0 = f(p) &= ((n-2)\alpha_0- n(n-2)\beta_0)(x_1-x_2) + ((n-2)\alpha_1 -n(n-2)\beta_1) p_1(x_1-x_2) \\ &- (n-1) \gamma_0 (x_1-x_2)^2 + 2(n-1) \gamma_0 (x_1-x_2)(x_1-x_3).
\end{align*}
We obtain $\alpha_0 = n \beta_0$ and at least one more non-trivial linear relation among $\alpha_1, \beta_1, \gamma_0$, hence $\dim_\CC(\Hom^{S_n}_P(I,P/I)) \leq 4$.

Finally, we need to show that $I$ lies in the smoothable component of $\HH_\rho$. For this, consider the vanishing ideal $J$ of the two $S_n$-orbits of $(b,a,a,\ldots,a)$ and $(d,c,c,\ldots,c)$ with $a,b,c,d \in \CC$ and $a \neq b$, $c \neq d$, $b \neq d$ but such that $b-a = d-c$. It can be checked that then the polynomial
\begin{equation*}
    -2 p_1(x_1-x_2) + n(x_1^2 - x_2^2) - (n-2)(b-a)(x_1-x_2)
\end{equation*}
lies in $J$. With this it is easy to see that all generators of $I$ lie in the associated graded ideal of $J$, hence the latter must agree with $I$. \qed

\subsection{Irreducibility in Table~\ref{tab:classification}}

\begin{proposition}
    Let $\rho$ be any of the $S_n$-representations in Table~\ref{tab:classification}. Then $\HH_\rho$ is irreducible. 
\end{proposition}

\begin{proof}
For rows~1,~3,~6 and~8--13 this follows from smoothness and connectedness in each case.

For the remaining rows~2,~4,~5~and~7, we have seen that all homogeneous symmetric ideals lie in the respective smoothable components, proving connectedness of $\HH_\rho$ in each case. It is therefore enough to consider the homogeneous symmetric ideals $I$ which are \emph{singular} points of $\HH_\rho$ and to prove, in each case, that all ideals whose associated graded ideal equals $I$ lie in the smoothable component again.

The three singular ideals in rows~2, 4 and~5 can all be written as $I = (x_i-x_j)_{\geq 2} + \mm^d$, where $d=2$ for row~4, $d=3$ for row~5, and $d>3$ for row~2.
Let now $J \subseteq P$ be any symmetric ideal whose associated graded ideal is $I$. For $d \geq 3$, we use
\begin{equation*}
    N_1 = \langle x_i^2 - x_j^2 \rangle \oplus \langle p_1(x_i - x_j)  \rangle \oplus \langle (x_i - x_j)(x_k - x_l) \rangle \oplus \langle p_1^2 - n p_2 \rangle \oplus \langle p_1^d \rangle
\end{equation*}
as generators for $I$. We claim that $J$ has generators of the form
\begin{gather*}
    (x_i - x_j)(x_k - x_l), \\
    p_1(x_i - x_j) + a_1(x_i - x_j), \\
    x_i^2 - x_j^2 + a_2(x_i - x_j), \\
    p_1^2 - np_2 + b_1 p_1 + b_2, \\
    p_1^d + \sum_{i=0}^{d-1} c_i p_1^i,
\end{gather*}
for some $a_1,a_2,b_1,b_2,c_i \in \CC$. For instance, since $(x_i - x_j)(x_k - x_l) \in I$ and $I$ is the associated graded ideal of $J$, there is an element $(x_i - x_j)(x_k - x_l) + \ell + c \in J$ for some linear form $\ell$ and some constant $c$. Applying the transposition $(i,j)$ to this polynomial yields $-(x_i - x_j)(x_k - x_l) + (i,j)(\ell) + c \in J$, hence the sum must lie in $J$ as well, and equals $\ell + (i,j)(\ell) + 2c$. It follows that $\ell = \lambda (x_i - x_j)$ for some $\lambda \in \CC$ and that $c = 0$. For the same reason, we also get $\ell = \mu (x_k - x_l)$ for some $\mu \in \CC$, hence $\ell = 0$. The reasoning for the other generators is similar.

Note that for all $k \geq 1$, we have $p_1^k(x_1 - x_2) \equiv -a_1 p_1^{k-1}(x_1-x_2) \pmod J$ and hence $p_1^k(x_1-x_2) \equiv (-a_1)^k(x_1-x_2) \pmod J$ for all $k \geq 0$.
Let $c_d \coloneqq 1$. Then we have
\begin{align*}
    J &\ni \left( \sum_{k=0}^d c_k p_1^k \right)(x_1-x_2) - (p_1(x_1-x_2) + a_1(x_1-x_2)) \left( \sum_{k=0}^{d-1} c_{k+1}p_1^k \right) \\
    &= c_0 (x_1-x_2) - a_1 \sum_{k=0}^{d-1} c_{k+1} p_1^k (x_1-x_2) \\
    &\equiv \left( (-a_1)^d + c_{d-1}(-a_1)^{d-1} + \ldots + c_0 \right) (x_1-x_2) \pmod J .
\end{align*}
Hence, we must have $(-a_1)^d + c_{d-1}(-a_1)^{d-1} + \ldots + c_0 = 0$.
Moreover,
\begin{align*}
    J &\ni \frac{-1}{(n-2)!} \sum_{\sigma \in S_n} \sigma (x_1(p_1(x_1-x_2) + a_1(x_1-x_2))) - p_1(p_1^2 - np_2 + b_1 p_1 + b_2) \\
    &= a_1 (p_1^2 - n p_2) - b_1 p_1^2 - b_2 p_1 \\
    &\equiv -b_1 p_1^2 - (a_1 b_1 + b_2)p_1 - a_1 b_2 \pmod J .
\end{align*}
Since $d \geq 3$, we obtain $b_1 = 0$ and hence also $b_2 = 0$. Together with the previous equation we obtain a flat family over an irreducible base isomorphic to an affine space of dimension $d+1$ (since we can eliminate the three variables $b_1,b_2,c_0$). Since the subset of the smoothable component consisting of all ideals whose associated graded ideal is $I$ has dimension $d+1$ as well, the general member of this family must be a smoothable ideal, hence \emph{every} member of the family is a smoothable ideal, proving our claim.

\medskip
For $d=2$ we use instead the monomial generators of $I = \mm^2$ and obtain the following generators for $J$:
\begin{gather*}
    x_i x_j + a_1 (x_i + x_j) + a_2 p_1 + a_3, \\
    x_i^2 + b_1 x_i + b_2 p_1 + b_3.
\end{gather*}
Since $\dim_\CC(P/J) = n+1$, the vanishing set of $J$ must contain at least one point on the diagonal. Hence, using the $\mathbb{G}_a$-action, we may assume that the origin is contained in the vanishing set of $J$, i.e., $a_3 = b_3 = 0$. But which parameters $a_1,a_2,b_1,b_2$ do then actually give us an ideal $J$ of the correct colength? Clearly, the associated graded ideal of $J$ contains $\mm^2$ for every choice of parameters. But equality need not hold in general; the obstruction are linear forms contained in $J$. These linear forms, however, can only arise from syzygies of the monomial generators of $\mm^2$ in degree~$3$, i.e., from the identities $x_i^2 x_j = x_i (x_i x_j)$ and $(x_i x_j) x_k = x_i (x_j x_k)$. Using these, we obtain all possible linear forms contained in $J$. Their coefficients yield $3$ linearly independent homogeneous polynomials $F_1,F_2,F_3$ of degree $2$ in $a_1,a_2,b_1,b_2$. The coefficients of the $F_i$, in turn, are integer constants or linear polynomials in $n$. Treating $n$ as a variable, we consider the ideal $(F_1,F_2,F_3)$ in the polynomial ring $\QQ[a_1,a_2,b_1,b_2,n]$. A computation in \texttt{Macaulay2} then provides the minimal primes of this ideal (over $\QQ$). They are $\mathfrak{p}_1 \coloneqq (a_1,b_1,a_2-b_2)$ and $\mathfrak{p}_2 \coloneqq (b_2+a_1+(n-1)a_2, a_1^2 + (n-2)a_1 a_2 + a_2 b_1)$. Both are geometrically prime, i.e., prime over $\CC$, even if $n$ is regarded as an integer again. The first prime ideal $\mathfrak{p}_1$ only parametrizes ideals defining a subscheme of $\CC^n$ for which the origin is \emph{not} reduced. Precisely, apart from $\mm^2$ itself, the members of the family defined by $\mathfrak{p}_1$ have the form $((p_1) + \mm^2) \cap \mm_{(c,c,\ldots,c)}$ for some $c \in \CC \setminus \{0\}$. We have already seen that these ideals all lie in the smoothable component.
The flat family defined by $\mathfrak{p}_2$ must then contain, in particular, the set of all symmetric ideals whose associated graded ideal is $\mm^2$ and which define a subscheme of $\CC^n$ for which the origin is a reduced point. Now, the \emph{general} radical ideal in $\HH_\rho$ has associated graded ideal equal to $\mm^2$, and the subset whose vanishing set contains the origin has dimension $2$. On the other hand, $\dim(\CC[a_1,a_2,b_1,b_2]/\mathfrak{p}_2) = 2$ as well. Therefore, by irreducibility, \emph{every} member of the family defined by $\mathfrak{p}_2$ is smoothable, concluding the argument.

\medskip
For the last remaining singular ideal $I = (p_1^2,p_2,p_1(x_i-x_j),(x_i-x_j)(x_k-x_l)) + \mm^3$ of Row~7(c) the reasoning is as follows. We assume first $n \geq 4$.
Let $J \subseteq P$ be any symmetric ideal with $\gr(J) = I$. We use
\[N_1 = \langle p_1^2 \rangle \oplus \langle p_2 \rangle \oplus \langle p_1(x_i-x_j) \rangle \oplus \langle (x_i-x_j)(x_k-x_l) \rangle \oplus \langle p_3 \rangle \]
as generators for $I$.
Then, clearly, $J$ has generators of the form 
\begin{align*}
    (x_i-x_j)(x_k-x_l), \\
    p_1^2 + A p_1 + B, \\
    p_2 + C p_1 + D, \\
    p_1(x_i-x_j) + E (x_i-x_j), \\
    p_3 + F p_1 + G.
\end{align*}
First, we note that $I$ is not the associated graded ideal of \emph{any} radical ideal. The reason is that the only possible radical ideals are the vanishing ideals of two distinct orbits represented by $(b,a,a,\ldots,a)$ and $(d,c,c,\ldots,c)$ for $a \neq b$, $c \neq d$. If $p_1(b,a,a,\ldots,a) = p_1(d,c,c,\ldots,c)$, then clearly $p_1$ is contained in the associated graded ideal but not in $I$. Otherwise the associated graded ideal contains some linear combination of $p_1(x_1-x_2)$ and $x_1^2 - x_2^2$ in which the latter occurs with a non-zero coefficient while $I$ only contains the former.

Next, we claim that the vanishing set of $J$ contains a diagonal point. To see this, note that the vanishing set of a symmetric ideal $J$ whose associated graded ideal is $I$, which is not radical and whose vanishing set does not contain a diagonal point must consist of precisely one orbit generated by $(b,a,a,\ldots,a)$ for $a \neq b$. Every point must have multiplicity $2$ by symmetry and since $\dim_\CC(P/J) = 2n$. The stabilizer of the subscheme of $\Spec(P/J)$ supported at a single point, say $(b,a,a,\ldots,a)$ is precisely $S_{n-1}$, embedded into $S_n$ as the permutations fixing the first index. Since its multiplicity is $2$, the subscheme supported at $(b,a,a,\ldots,a)$ is defined by the ideal $((x_1-b)^2, x_2 - a, \ldots, x_n - a)$, so that $J$ is the intersection of all permutations of this ideal. However, this is impossible since the associated graded ideal then would not contain $p_1(x_1-x_2)$ as can be checked after reducing to $a=0$ via the additive group action.

Hence, we may assume $V(J)$ contains the origin, so that all constant terms of the above generators are zero. If now $E \neq 0$, then localizing at $p_1 + E$ shows that $(x_i - x_j)$ is contained in the localized ideal, so this ideal corresponds to Row~1, contradicting the fact that $P/J \cong \rho$. Hence, $E = 0$. If $A \neq 0$, we can localize at $p_1 + A$ and obtain the ideal generated by $(x_i-x_j)(x_k-x_l)$ and $p_1,p_2,p_3$ which is contained in the ideal of row~6; this is again impossible. Hence, $A = 0$ as well. The only parameters we are left with are $C$ and $F$. This proves that the subset of $\HH_\rho$ of ideals $J$ with $\gr(J) = I$ and whose vanishing set contains the origin is either irreducible of dimension~$2$ or has dimension $\leq 1$. Therefore, taking the additive group action back in, the subset of $\HH_\rho$ of ideals $J$ with $\gr(J) = I$ (without conditions on $V(J)$) is either irreducible of dimension~$3$ or has dimension $\leq 2$.

Let now $H$ be the locally closed subset of $\HH_\rho$ of all ideals having the same Hilbert function $h$ as $I$. We endow $H$ with the reduced scheme structure. By Subsection~\ref{subsec:assocGraded}, there is a natural map $\varphi \colon H \longrightarrow \Hilb^{S_n \times \CC^\ast}_{\rho,h}(\CC^n)$, sending an ideal $I$ to its associated graded ideal $\gr(I)$. Note that the target of $\varphi$ consists precisely of the ideals in Row~7(c), so it is irreducible of dimension~$1$. Moreover, all fibers of $\varphi$ over closed points corresponding to homogeneous ideals $I'$ have a component of dimension at least $3$ containing $I'$. Indeed, for any such ideal in Row~7(c) other than $I$ itself consider the vanishing ideal $J$ of two distinct orbits represented by $(b,a,a,\ldots,a)$ and $(d,c,c,\ldots,c)$, $a \neq b$, $c \neq d$. If $p_1(b,a,a,\ldots,a) \neq p_1(d,c,c,\ldots,c)$, then there is a unique $\alpha \in \CC$ such that $\alpha p_1(x_1 - x_2) + (x_1^2 - x_2^2) + \beta (x_1 - x_2) \in J$ for some $\beta \in \CC$. This shows that $\gr J$ belongs to Row~7(c).
Conversely, given any $\alpha \in \CC$, the existence of $\beta$ such that the above polynomial belongs to $J$ defines exactly one linear condition on $a,b,c,d$. Observe that this means that the fiber of $\varphi$ over $\gr(J)$ contains a $3$-dimensional subset of the smoothable component containing $\gr(J)$.

Let $Z$ be the smoothable component of $\HH_\rho$. Consider the composition
\begin{equation*}
    (Z \cap H)_{\mathrm{red}} \hookrightarrow H \overset{\varphi}{\longrightarrow} \Hilb^{S_n \times \CC^\ast}_{\rho,h}(\CC^n).
\end{equation*}
By the upper semicontinuity of fiber dimensions (on the source!), the intersection $\varphi^{-1}([I]) \cap Z$ has dimension at least $3$. This proves, first, that $\varphi^{-1}([I])$ is actually irreducible of dimension $3$ and, second, that this entire fiber $\varphi^{-1}([I])$ lies in the smoothable component $Z$. This in turn proves the irreducibility of $\HH_\rho$ for $n \geq 4$.

For $n = 3$, the role of $(x_i-x_j)(x_k-x_l)$ is played by $(x_1-x_2)(x_1-x_3)(x_2-x_3)$. With this replacement all arguments stay the same.
\end{proof}

\section{Conclusion}\label{section:conclusion}

\subsection{Connectedness, irreducibility, singularities}
The last two properties are in general not well-understood even for usual Hilbert schemes of points, and so the same is true in the invariant setting. For general reductive groups, invariant Hilbert schemes may be disconnected, see \cite{BrionInvariantHilb} and the references therein. On the other hand, all examples we have seen in this article are connected.

\begin{question}
    Is $\HH_\rho$ connected whenever it is non-empty, at least if $\rho$ is a direct sum of permutation modules?
\end{question}

\begin{question}
    What is a minimal example of $\rho$ such that $\HH_\rho$ is reducible?
\end{question}

As a consequence of the next proposition, reducible examples exist, for example, in the case where $\rho = \CC[S_n]^l$ is a direct sum of $l$ regular representations if $n$ and $l$ are large enough.
\begin{proposition}\label{prop:HilbChowIso}
    If $\rho = \CC[S_n]^{\oplus l}$, the Hilbert--Chow morphism
    \begin{equation*}
    \gamma \colon \HH_{\rho} \longrightarrow \Hilb_{l}(\CC^n/S_n)
\end{equation*}
is an isomorphism.
\end{proposition}

\begin{proof}
    There are inverse natural transformations between the $S_n$-invariant Hilbert functor of $\CC^n$ for $\rho = \CC[S_n]^{\oplus l}$ and the usual Hilbert functor of $\CC^n/S_n$ for length~$l$ subschemes. By \cite[Section~3.4]{BrionInvariantHilb}, for a finite type $\CC$-scheme $T$, the natural transformation corresponding to the Hilbert--Chow morphism
    \begin{equation*}
        \mathcal{H}ilb^{S_n}_{\CC[S_n]^{\oplus l}}(\CC^n)(T) \rightarrow \mathcal{H}ilb_l(\CC^n/S_n)(T)
    \end{equation*}
    sends an $S_n$-stable quasi-coherent ideal sheaf $\mathcal{I} \subseteq \mathcal{O}_T[x_1,\ldots,x_n]$ to its invariant subsheaf
    $$\mathcal{I}^{S_n} \subseteq \mathcal{O}_T[p_1,\ldots,p_n].$$ (Here, the functors in the first display denote the invariant and the usual Hilbert functor, respectively.) The inverse natural transformation sends a quasi-coherent ideal sheaf $\mathcal{J} \subseteq \mathcal{O}_T[p_1,\ldots,p_n]$ to
    $$\mathcal{J} \otimes_{\mathcal{O}_T[p_1,\ldots,p_n]} \mathcal{O}_T[x_1,\ldots,x_n] \hookrightarrow \mathcal{O}_T[x_1,\ldots,x_n].$$
    This map is well-defined because $\mathcal{O}_T[x_1,\ldots,x_n]$ is a free module over $\mathcal{O}_T[p_1,\ldots,p_n]$ of rank $n!$. More precisely, the quotient
    \begin{equation*}
        \left( \mathcal{O}_T[p_1,\ldots,p_n] / \mathcal{J} \right) \otimes_{\mathcal{O}_T[p_1,\ldots,p_n]} \mathcal{O}_T[x_1,\ldots,x_n]
    \end{equation*}
    has the correct isotypic decomposition since $\mathcal{O}_T[p_1,\ldots,p_n] / \mathcal{J}$ is locally free of rank $l$.
    In order to prove that the two maps are inverse to each other, observe that
    $$\mathcal{J} \subseteq \left(\mathcal{J} \otimes_{\mathcal{O}_T[p_1,\ldots,p_n]} \mathcal{O}_T[x_1,\ldots,x_n]\right)^{S_n}$$
    and
    $$\mathcal{I}^{S_n} \otimes_{\mathcal{O}_T[p_1,\ldots,p_n]} \mathcal{O}_T[x_1,\ldots,x_n] \subseteq \mathcal{I}.$$
    The two natural surjections induced by these inclusions are surjections of locally free $\mathcal{O}_T$-modules of the same finite rank and hence isomorphisms, proving equality.
    \end{proof}

In Proposition~\ref{prop:HilbChowIso}, the target is reducible for every $n \geq 3$ if $l$ is sufficiently large by \cite{Iarrobino1972Reducibility}, hence so is the source. Another series of examples with precisely two irreducible components is given by $n \geq 4$ and $l = 8$ by~\cite{Cartwright2009HilbertSchemeOf8Points}.

\subsection{Stabilization}\label{subsec:stabilization}
Given an $S_n$-representation $\rho = \bigoplus_{i} w_i S^{\lambda^i}$, and $k \geq 0$, we let $\rho(k)$ be the $S_{n+k}$-representation obtained by replacing each $\lambda^i = (\lambda^i_1,\ldots,\lambda^i_{m_i}) \vdash n$ by $\lambda^i(k) \coloneqq (\lambda^i_1 + k, \lambda^i_2,\ldots,\lambda^i_{m_i}) \vdash n+k$.

\begin{question}\label{question:stabilization}
    Is it true that for every $S_n$-representation $\rho$ there exists $k_0$ such that for all $k \geq k_0$ we have
    \begin{equation*}
        \HH_{\rho(k)} \cong \HH_{\rho(k+1)}?
    \end{equation*}
\end{question}

A look at Table~\ref{tab:classification} seems to indicate that for all large enough $n$, intuitively, the homogeneous ideals in $\HH_\rho$ correspond bijectively to the homogeneous ideals in $\HH_{\rho(1)}$, and corresponding ideals have similar properties; for instance, the dimensions of the tangent spaces are the same. However, it does not seem to be straightforward to define natural maps $\HH_\rho \rightarrow \HH_{\rho(1)}$ or $\HH_{\rho(1)} \rightarrow \HH_\rho$ at the functorial level. So, while Question~\ref{question:stabilization} might be too much to ask for in this generality, the similarities between $\HH_\rho$ and $\HH_{\rho(1)}$ in all examples we have seen in this paper call for an explanation.

\subsection*{Acknowledgments} The second named author thanks Piotr Oszer for an inspiring discussion that eventually led to Proposition~\ref{prop:reducedPoint} and Remark~\ref{rmk:noPositiveTangents}. The largest part of this work was carried out while both authors were based at OvGU Magdeburg. There, the authors were supported by the DFG – 314838170, GRK 2297 MathCoRe. The final part was finished when the first named author was based at TU Chemnitz and the second named author was based at the MPI MiS Leipzig and HU Berlin.

\bibliographystyle{alpha}
\bibliography{library.bib}

\begin{thebibliography}{CEVV09}

\bibitem[AB05]{AlexeevBrion}
Valery Alexeev and Michel Brion.
\newblock Moduli of affine schemes with reductive group action.
\newblock {\em J. Algebraic Geom.}, 14(1):83--117, 2005.

\bibitem[AH07]{Aschenbrenner2007Finite}
Matthias Aschenbrenner and Christopher Hillar.
\newblock Finite generation of symmetric ideals.
\newblock {\em Trans. Amer. Math. Soc.}, 359(11):5171--5192, 2007.

\bibitem[ATY97]{Ariki1997Higher}
Susumu Ariki, Tomohide Terasoma, and Hiro-Fumi Yamada.
\newblock Higher {S}pecht polynomials.
\newblock {\em Hiroshima Math. J.}, 27(1):177--188, 1997.

\bibitem[BG92]{BergeronGarsia1992Harmonic}
N.~Bergeron and A.~M. Garsia.
\newblock On certain spaces of harmonic polynomials.
\newblock In {\em Hypergeometric functions on domains of positivity, {J}ack
  polynomials, and applications ({T}ampa, {FL}, 1991)}, volume 138 of {\em
  Contemp. Math.}, pages 51--86. Amer. Math. Soc., Providence, RI, 1992.

\bibitem[Bri13]{BrionInvariantHilb}
Michel Brion.
\newblock Invariant {H}ilbert schemes.
\newblock In {\em Handbook of moduli. {V}ol. {I}}, volume~24 of {\em Adv. Lect.
  Math. (ALM)}, pages 64--117. Int. Press, Somerville, MA, 2013.

\bibitem[CEF15]{ChurchEllenbergFarb2015FImodules}
Thomas Church, Jordan~S. Ellenberg, and Benson Farb.
\newblock F{I}-modules and stability for representations of symmetric groups.
\newblock {\em Duke Math. J.}, 164(9):1833--1910, 2015.

\bibitem[CEVV09]{Cartwright2009HilbertSchemeOf8Points}
Dustin~A. Cartwright, Daniel Erman, Mauricio Velasco, and Bianca Viray.
\newblock Hilbert schemes of 8 points.
\newblock {\em Algebra Number Theory}, 3(7):763--795, 2009.

\bibitem[CF13]{ChurchFarb2013RepresentationStability}
Thomas Church and Benson Farb.
\newblock Representation theory and homological stability.
\newblock {\em Adv. Math.}, 245:250--314, 2013.

\bibitem[Coh67]{Cohen1967Laws}
Daniel Cohen.
\newblock On the laws of a metabelian variety.
\newblock {\em J. Algebra}, 5:267--273, 1967.

\bibitem[Dra14]{Draisma2014Noetherianity}
Jan Draisma.
\newblock Noetherianity up to symmetry.
\newblock In {\em Combinatorial algebraic geometry}, volume 2108 of {\em
  Lecture Notes in Math.}, pages 33--61. Springer, Cham, 2014.

\bibitem[FH91]{FultonHarris1991ReprTheory}
William Fulton and Joe Harris.
\newblock {\em Representation theory}, volume 129 of {\em Graduate Texts in
  Mathematics}.
\newblock Springer-Verlag, New York, 1991.
\newblock A first course, Readings in Mathematics.

\bibitem[Fog73]{Fogarty1973Fixed}
John Fogarty.
\newblock Fixed point schemes.
\newblock {\em Amer. J. Math.}, 95:35--51, 1973.

\bibitem[Ful97]{Fulton}
William Fulton.
\newblock {\em Young tableaux}, volume~35 of {\em London Mathematical Society
  Student Texts}.
\newblock Cambridge University Press, Cambridge, 1997.
\newblock With applications to representation theory and geometry.

\bibitem[GP92]{GarsiaProcesi1992}
A.~M. Garsia and C.~Procesi.
\newblock On certain graded {$S_n$}-modules and the {$q$}-{K}ostka polynomials.
\newblock {\em Adv. Math.}, 94(1):82--138, 1992.

\bibitem[Gri22]{griffin2022orbit}
Sean~T Griffin.
\newblock Orbit harmonics for the union of two orbits.
\newblock {\em
  \href{https://arxiv.org/pdf/2204.04566.pdf}{{\color{blue}{arXiv:2204.04566}}}},
  2022.

\bibitem[Hai03]{Haiman2003Combinatorics}
Mark Haiman.
\newblock Combinatorics, symmetric functions, and {H}ilbert schemes.
\newblock In {\em Current developments in mathematics, 2002}, pages 39--111.
  Int. Press, Somerville, MA, 2003.

\bibitem[HS12]{Hillar2012Finite}
Christopher Hillar and Seth Sullivant.
\newblock Finite {G}r\"{o}bner bases in infinite dimensional polynomial rings
  and applications.
\newblock {\em Adv. Math.}, 229(1):1--25, 2012.

\bibitem[Iar72]{Iarrobino1972Reducibility}
A.~Iarrobino.
\newblock Reducibility of the families of {$0$}-dimensional schemes on a
  variety.
\newblock {\em Invent. Math.}, 15:72--77, 1972.

\bibitem[LNNR20]{NagelRoemer2020CodimUpToSymmetry}
Dinh~Van Le, Uwe Nagel, Hop~D. Nguyen, and Tim R\"{o}mer.
\newblock Codimension and projective dimension up to symmetry.
\newblock {\em Math. Nachr.}, 293(2):346--362, 2020.

\bibitem[LNNR21]{NagelRoemer2021RegularityUpToSymmetry}
Dinh~Van Le, Uwe Nagel, Hop~D. Nguyen, and Tim R\"{o}mer.
\newblock Castelnuovo-{M}umford regularity up to symmetry.
\newblock {\em Int. Math. Res. Not. IMRN}, (14):11010--11049, 2021.

\bibitem[LT15]{LehnTerpereau}
Christian Lehn and Ronan Terpereau.
\newblock Invariant deformation theory of affine schemes with reductive group
  action.
\newblock {\em J. Pure Appl. Algebra}, 219(9):4168--4202, 2015.

\bibitem[MOY22]{Murai2022Specht}
Satoshi Murai, Hidefumi Ohsugi, and Kohji Yanagawa.
\newblock A note on the reducedness and {G}r\"{o}bner bases of {S}pecht ideals.
\newblock {\em Comm. Algebra}, 50(12):5430--5434, 2022.

\bibitem[MRV21]{moustrou2021symmetric}
Philippe Moustrou, Cordian Riener, and Hugues Verdure.
\newblock Symmetric ideals, {S}pecht polynomials and solutions to symmetric
  systems of equations.
\newblock {\em J. Symbolic Comput.}, 107:106--121, 2021.

\bibitem[MS05]{Miller2005CombCommAlg}
Ezra Miller and Bernd Sturmfels.
\newblock {\em Combinatorial commutative algebra}, volume 227 of {\em Graduate
  Texts in Mathematics}.
\newblock Springer-Verlag, New York, 2005.

\bibitem[NC19]{nino2019algorithmic}
Jonathan~Andr{\'e}s Ni{\~n}o~Cort{\'e}s.
\newblock Algorithmic calculation of invariant rings under the action of finite
  sub-(groups): an implementation in {M}acaulay2.
\newblock Master's thesis, Universidad de los Andes, 2019.

\bibitem[NR17]{Nagel2017Equivariant}
Uwe Nagel and Tim R\"{o}mer.
\newblock Equivariant {H}ilbert series in non-noetherian polynomial rings.
\newblock {\em J. Algebra}, 486:204--245, 2017.

\bibitem[NR19]{Nagel2019FIvaryingCoeff}
Uwe Nagel and Tim R\"{o}mer.
\newblock F{I}- and {OI}-modules with varying coefficients.
\newblock {\em J. Algebra}, 535:286--322, 2019.

\bibitem[NS21]{nagpal2021symmetric}
Rohit Nagpal and Andrew Snowden.
\newblock Symmetric ideals of the infinite polynomial ring.
\newblock {\em
  \href{https://arxiv.org/abs/2107.13027}{{\color{blue}{arXiv:2107.13027}}}},
  2021.

\bibitem[RTAL13]{Riener2013Symmetries}
Cordian Riener, Thorsten Theobald, Lina~Jansson Andr\'{e}n, and Jean~B.
  Lasserre.
\newblock Exploiting symmetries in {SDP}-relaxations for polynomial
  optimization.
\newblock {\em Math. Oper. Res.}, 38(1):122--141, 2013.

\bibitem[Sag01]{Sagan}
Bruce~E. Sagan.
\newblock {\em The symmetric group}, volume 203 of {\em Graduate Texts in
  Mathematics}.
\newblock Springer-Verlag, New York, second edition, 2001.
\newblock Representations, combinatorial algorithms, and symmetric functions.

\bibitem[Spe37]{Specht}
Wilhelm Specht.
\newblock Zur {D}arstellungstheorie der symmetrischen {G}ruppe.
\newblock {\em Math. Z.}, 42(1):774--779, 1937.

\bibitem[Tan82]{Tanisaki1982Defining}
Toshiyuki Tanisaki.
\newblock Defining ideals of the closures of the conjugacy classes and
  representations of the {W}eyl groups.
\newblock {\em Tohoku Math. J. (2)}, 34(4):575--585, 1982.

\bibitem[TY93]{Terasoma1993Higher}
Tomohide Terasoma and Hirofumi Yamada.
\newblock Higher {S}pecht polynomials for the symmetric group.
\newblock {\em Proc. Japan Acad. Ser. A Math. Sci.}, 69(2):41--44, 1993.

\end{thebibliography}
\end{document}